\documentclass[a4paper,12pt]{amsart}
\usepackage{hyperref}
\usepackage{geometry}
\usepackage[british]{babel}
\usepackage{wrapfig}
\usepackage[comma, square, numbers]{natbib}
\usepackage{yfonts}
\usepackage[utf8]{inputenc}
\usepackage{amssymb}
\usepackage[normalem]{ulem}
\usepackage{amsthm}
\usepackage{graphics}
\usepackage{amsmath}
\usepackage{amsthm}
\usepackage{dirtytalk}
\usepackage{amstext}
\usepackage{subfigure}
\usepackage{engrec}
\usepackage{floatflt}
\usepackage{rotating}
\usepackage[safe,extra]{tipa}
\usepackage[font={footnotesize}]{caption}
\usepackage{framed}
\usepackage{listings}
\usepackage[dvipsnames]{xcolor}
\usepackage{lipsum}
\usepackage[titletoc,title]{appendix}

\newcommand{\mc}{\mathcal}
\newcommand{\mb}{\mathbb}

\newcommand{\R}{\mb R}

\newcommand{\N}{\mb N}
\newcommand{\Z}{\mb Z}

\newcommand{\T}{\mb T}

\newcommand{\eea}{\end{align}}

\renewcommand{\epsilon}{\varepsilon}
\renewcommand{\bar}{\overline}
\renewcommand{\tilde}{\widetilde}

\newcommand{\bo}{\boldsymbol}
\renewcommand{\phi}{\varphi}
\renewcommand{\angle}{\measuredangle}

\renewcommand\upsilon{\theta}

\usepackage[normalem]{ulem}

\newtheorem{theorem}{Theorem}[section]

\newtheorem{corollary}[theorem]{Corollary}

\newtheorem{lemma}[theorem]{Lemma}
\newtheorem{proposition}[theorem]{Proposition}
\theoremstyle{definition}
\newtheorem{definition}[theorem]{Definition}

\newtheorem{remark}[theorem]{Remark}

\newtheoremstyle{algorithm}
{4pt}
{4pt}
{}
{}
{}
{:}
{\newline}
{}

\usepackage{xcolor}

\newtheorem{algorithm}{Algorithm}

\newcommand{\balgorithm}{\begin{algorithm}\begin{framed}\ }
\newcommand{\ealgorithm}{\end{framed}\end{algorithm}}

\newcommand{\pippo}{\begin{code} \  \begin{lstlisting}[frame=single]}
\newcommand{\pappo}{\end{lstlisting} \end{code}}

\newcommand{\bd}{\begin{definition}}
\newcommand{\ed}{\end{definition}}

\newcommand{\bt}{\begin{theorem}}
\newcommand{\et}{\end{theorem}}
\newcommand{\bp}{\begin{proposition}}
\newcommand{\ep}{\end{proposition}}
\newcommand{\bc}{\begin{corollary}}
\newcommand{\ec}{\end{corollary}} 
\newcommand{\bl}{\begin{lemma}}
\newcommand{\el}{\end{lemma}}
\newcommand{\br}{\begin{remark}}

\newcommand{\er}{\end{remark}}

\usepackage{enumitem}

\title{{Existence of physical measures in some Excitation-Inhibition Networks}}

\author{Matteo Tanzi} 
\address{Matteo Tanzi: Courant Institute of Mathematical Sciences, New York University, New York, NY 10012, USA}
\email{matteo.tanzi@nyu.edu}
\thanks{MT acknowledges funding from the H2020 Marie S{\l}odowska-Curie Actions, project: ``Ergodic Theory of Complex Systems" (project no. 843880)}
\author{Lai-Sang Young}
\address{Lai-Sang Young: Courant Institute of Mathematical Sciences, New York University, New York, NY 10012, USA, and Institute for Advanced Study, Princeton, New Jersey 08540, USA}
\email{lsy@cims.nyu.edu}
\thanks{LSY was supported in part by NSF Grant DMS 1901009}

\begin{document}

\maketitle

\begin{abstract}  In this paper we present a rigorous analysis of a class of coupled dynamical systems
in which two distinct types of components, one excitatory and the other inhibitory, interact
with one another. These network models  are finite in size but can be arbitrarily large. 
They are inspired by real biological networks, and possess features that are idealizations of those in 
biological systems. Individual components of the network are represented by simple, much studied
dynamical systems. Complex dynamical patterns on the network level emerge as a result 
of the coupling among its constituent subsystems.
Appealing to existing techniques in (nonuniform) hyperbolic theory, we study their Lyapunov
exponents and entropy, and prove that large time network dynamics are governed by 
physical measures with the SRB property.
\end{abstract}
{

\section{Introduction}
As dynamical systems come in too many flavors to be classified 
or even systematically described, when studying the subject one usually learns from paradigms. 
In the mathematical theory of chaotic systems, a great deal of intuition 
has been derived from classical examples such as expanding circle maps \cite{Szlenk},
geodesic flows on manifolds of negative curvature \cite{hopf51statistik,anosov1967geodesic}, hyperbolic billiards \cite{Sinai1,BunimovichErgodicProp} \cite{bunimovich1990markov,chernov2006chaotic}, the Lorenz attractor \cite{lorenz1962statistical,guckenheimer1979structural}, logistic maps \cite{jakobson1981absolutely,collet1980abundance,benedicks1985iterations}, 
H\'enon attractors and generalizations \cite{benedicks1991dynamics,young1993sinai,wang2008toward}, and so on. Our notion of what a chaotic dynamical system 
looks like has been intimately tied to these examples. Textbook pictures of
chaotic behavior are those with identifiable expanding and contracting 
directions in their phase spaces, and where most pairs of nearby orbits 
can be seen to diverge quickly, at exponential rates.

Chaotic behavior in the sense of hyperbolicity in high dimensional 
systems can have a different appearance, however. Independent of
dimension, the positivity of Lyapunov exponents must, by definition, 
entail the same local geometry of invariant subspaces and invariant 
manifolds, but visualizing this picture for a system with high dimensional phase space
 can be challenging. 
For high dimensional systems, it may be profitable to turn to other more 
salient characteristics. 
For example, in a network in which a large number of simpler constituent 
dynamical systems are coupled together, dynamics of the constituent 
subsystems may be more readily observable. Properties such as degree 
of synchronization, the coherence or incoherence of dynamical behavior among constituent
subsystems, the number of distingishable states and the predictability of dynamical patterns --- all are 
known to be of consequence in real-world examples, see e.g. \cite{buck1966biology,traub1982cellular,kuramoto2003chemical,blaabjerg2006overview}. 
But how are these behaviors related to hyperbolicity?
What are the signatures of chaotic behaviors in networks? How is the number of
positive Lyapunov exponents reflected? And do concepts like SRB
measures and physical measures continue to make sense? These and other
questions are part of a broad swath of open territory in the theory  of high dimensional 
dynamical systems that remains to be explored, and are most likely going to be (at least partially) answered through the introduction of new paradigms for high-dimensional systems.

In the belief that examples will contribute to our understanding of hyperbolicity in
high dimensional systems, we present in this paper { a new class of models  --- 
new in the sense that their ergodic theory has not been studied ---
and demonstrate how they can be analyzed using existing tools. 
These models are examples of Excitation-Inhibition networks.
They consist of} arbitrary numbers of interacting components some of which excite and
others suppress. The interaction is bi-directional and the model is more general
than a skew-product: an excitatory subsystem excites inhibitory components 
which when activated send feedback inhibition to the excitatory subsystem. 
These examples are inspired by biology, but we do not 
pretend  that they are depictions of any specific biological system. In order to make 
connections with hyperbolic theory, it was necessary to make the models analyzable, 
and to do that, we have had to introduce a fair amount of idealization.

{ Though far from the first such examples to be studied, we mention 
two ways in which our models differ from many previous works.
The first is that they}
include interactions of components that are not identical.
While homogeneous materials in physics (e.g. \cite{spohn1977stationary}) have 
inspired the study of systems in which identical maps are coupled (for a random sample see 
\cite{BS88,K93,BK96,kobre2007extended,keller2009map,SB16} 
and the articles in \cite{CF05}),
examples from biological networks typically 
involve different substances or agents with distinct characteristics and functions 
(see { \cite{breitsch2015cell,fernandez2014typical,mauroy2008clustering,wilson2015clustered,young2012clustering, lin2020origin}}; see also \cite{PerTanzHeterCoupl}). 
{ The network models studied here are an addition to the relatively small collection
of examples of the latter kind.}
A second characteristic of our example networks is that their hyperbolicity is not 
apparent; that is, expanding and contracting directions are not obvious to the eye, 
unlike, e.g., weakly coupled lattices of expanding circle maps, 
as in e.g. \cite{BS88,KL06} or even more strongly coupled networks, as in e.g. \cite{KY10}.
It is our hope that examples in which hyperbolic properties are emergent may offer greater insight
into how chaotic (or nonchaotic) behaviors would manifest in these large dynamical systems.

Finally, though the mathematical models studied here are not intended to be models
of realistic physical or biological systems, they are inspired by real biological phenomena. We have 
incorporated into our models features such as the competition and balancing of excitatory 
and inhibitory forces, the activation of certain processes upon 
threshold crossing, and relaxation to equilibria in the absence of inputs.  
This paper is a small step in the direction of promoting biology-inspired models, which are,
in our view, underrepresented among the current pool of examples in dynamical systems.
More generally, we wish to demonstrate that dynamical systems theory 
has the language and the tools to shed light on natural phenomena, to offer insight on
a conceptual level even when  analysis of the system exactly as defined
is out of reach. 
\section{Model Description and Main Results}\label{Sec:Modelandmainres}

Sect \ref{Sec:Overview} contains an informal overview of the model; details
are given in Sect \ref{Sec:PrecModelDesc}. Our main results are stated in Sect \ref{Sec:Mainresults} below.

\subsection{Overview}\label{Sec:Overview}

{The network studied in this paper} is made up of arbitrary numbers of
{\it excitatory} and {\it inhibitory} components. Simplifying, we group the excitatory 
components together to form a single ``excitatory environment" (E-environment).  Inhibitory 
components (I-components) are modeled individually.

When decoupled from the rest of the network, the E-environment 
is modeled as an Anosov flow that we take as exemplification of chaotic behavior in continuous time.  A global cross-section $\sigma_0$ to this flow is fixed.
Excitatory signals are sent to 
I-components each time the trajectory returns to $\sigma_0$.

When decoupled from the E-environment, I-components are modeled 
as North-South flows on the unit circle, i.e., there are two fixed points, to be
thought of as N and S-poles; all nonstationary trajectories go from N-pole to S-pole. 
Each excitatory signal produces a fairly abrupt rotation of the circle. For each I-component, 
the amount rotated depends on the location of $\sigma_0$ visited by the Anosov  trajectory;
consecutive visits produce what  resembles a sequence of random rotations. 
Different I-components are
affected differently at each return to $\sigma_0$.

\begin{figure}
\center
\includegraphics[scale=0.5]{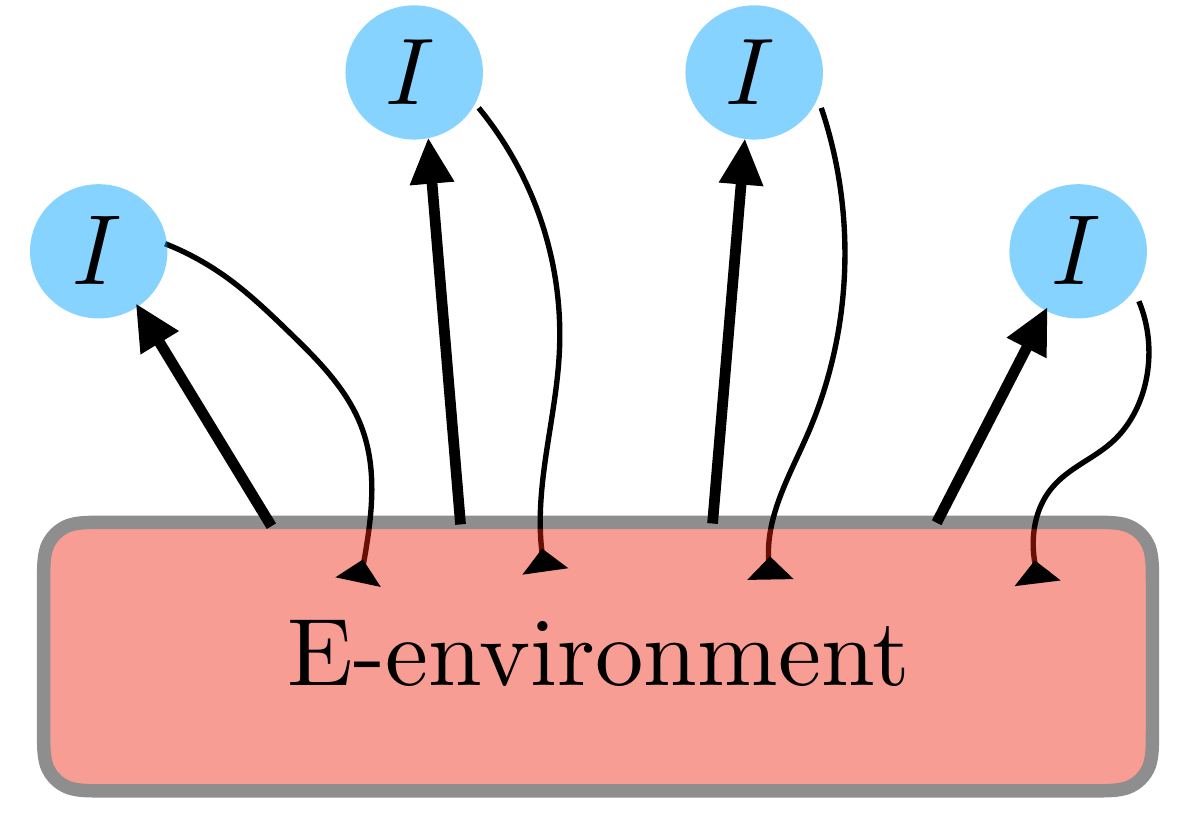}
\caption{Network architecture: An Excitatory (E) environment interacting with an arbitrary number of Inhibitory (I) components. When activated, the I-components slow down the dynamics of the E-environment.}
\label{Fig:netarchitect}
\end{figure}

Finally, there is feedback from the I-components to the E-environment. For each I-component,
there is a designated region of its phase space with the property that while there, the component has
a suppressive effect on the E-environment, modeled as a reduction in speed for the Anosov
flow. The I-components do not interact with one another directly, but only through the
E-environment: collectively they determine when the next excitatory signal will be.


This completes our model description in words. A precise description is given below.
As we will show, our model is a piecewise smooth flow on a $(3+N)$-dimensional phase space
where $N \in \Z^+$ is the number of I-components.
Our main result is that the dynamics of {networks of this type can be} described by 
natural invariant measures that are physical measures 
with the SRB property; see {\cite{eckmann1985ergodic}}. {Numerical
illustrations are presented in} {Sect. \ref{Sec:RelModNumSim}}.

\subsection{Precise model description}\label{Sec:PrecModelDesc}

For $k \in \Z^+$, let $\T^k := \R^k/\Z^k$ denote the $k$-dimensional torus, which we
also identify with $\T \times \cdots \times \T$, the $k$-fold product of $\T=\R/\Z$. 
Euclidean norms on $\T^k,  \R$ and products thereof are denoted by $| \cdot |$.
Below we present the full model in four steps (A-D), describing first the dynamics
of the E-environment and I-components in isolation, i.e when each is decoupled from 
the rest of the system, before describing how they interact.

\bigskip \noindent
A. {\it Excitatory environment in isolation.} Consider a flow $f^t$ that is a time-$1$-suspension 
of a linear hyperbolic toral automorphism $A: \T^2 \circlearrowleft$.

More precisely, let $A: \T^2 \circlearrowleft$ be an Anosov diffeomorphism, which 
for simplicity we assume to be linear, i.e. there is a constant splitting 
$T\T^2=E_A^u\oplus E_A^s$ 
such that for all $p \in \T^2$, 
$$DA_pE_A^u=E_A^u, \qquad DA_pE_A^s=E_A^s; $$
and there exists $\lambda>0$ such that 
\[
|DA_pv|=e^\lambda |v|\quad\forall v\in E^u_A,\quad\quad
|DA_pv|=e^{-\lambda}|v|\quad\forall v\in E^s_A.
\]

Letting $(x,y,w)$ denote the coordinates in $\T^2 \times [0,1]$, the flow $f^t$ is defined
on the  set $M_f :=(\T^2 \times [0,1])/\sim$ where $(x,y,1)$ is identified with $(A(x,y),0)$.
To introduce a Riemannian metric on $M_f$, we assign to the tangent space at
$(x,y,w) \in \T^2 \times [0,1)$ an inner product with respect to which $E^u_A, E^s_A$ 
and the $w$-axis are orthogonal and the resulting norm $\|\cdot\|$ has
$\|\cdot\| = |\cdot|$ in the $w$-direction  and 
$$
\|v\| = e^{\lambda w} |v| \ \ \mbox{for } v \in E^u_A, \qquad 
\|v\| = e^{-\lambda w} |v| \ \ \mbox{for } v \in E^s_A\ .
$$
It is easy to see that 
\begin{itemize}
\item[(i)] $M_f$ equipped with $\| \cdot \|$ is a $C^\infty$ Riemannian manifold; 
\item[(ii)]  if $f^t$ is the flow generated by the {vector field $\partial_w$}, the positive unit vector
in the $w$-direction, then $f^t$ is a $C^\infty$ Anosov flow for which $\|Df^t(v)\| = e^{\lambda t} \|v\|$
for $v \in E^u_A$ and $\|Df^t(v)\| = e^{-\lambda t} \|v\|$ for $v \in E^s_A$ everywhere on $M_f$.
\item[(iii)] Finally, $\sigma_0 := \T^2 \times \{0\}$ is a cross-section to the flow with
first return map $=A$ and return time $\equiv 1$.
\end{itemize}
 The E-environment of our model when operating in isolation is given by the flow 
 $f^t$ on $(M_f, \|\cdot\|)$.
 
\bigskip \noindent
B. {\it Inhibitory components when decoupled from the E-environment.}  
Fix arbitrary $N \in \Z^+$. For $i = 1,2, \cdots, N$,  we let $g^t_i$ be the flow on $\T$
generated by a $C^2$ vector field $v_{g_i}$ with the following properties: 
\begin{itemize}
\item[(i)] $v_{g_i}(0)=v_{g_i}\left(\frac12\right)=0$, and $v_{g_i}(z) \neq 0$ for $z \ne 0, \frac12$; 
\item[(ii)] $v_{g_i}'(0)=\lambda_{i}^{-}$ and $ v_{g_i}'\left(\frac12\right)=\lambda_{i}^{+}
$ \ \ for some $\lambda_{i}^{-} < 0 < \lambda_{i}^{+}$. 
\end{itemize}
It follows from (i) and (ii) above that $g^t_i$ is a North-South (N-S) flow, with S-pole at $z_i=0$
and N-pole at $z_i=\frac12$. All other trajectories go from N-pole to S-pole.

\bigskip
The {\bf phase space} of the full model is $M := M_f  \times \T^N$.

\bigskip \noindent
C. {\it Action of E-environment on I-components without feedback.}  
For each $i \in \{1, \cdots, N\}$, we let $r_i:\T \rightarrow \T$ be a $C^2$ order-preserving 
local diffeomorphism of degree $ \ge 1$
and, {identifying $\T$ with $[0,1)$, we let $\hat r_i : [0,1) \to \R$ be the lift of $r_i$ with $\hat r_i(0) \in [0,1)$, 
i.e., $\hat r_i$ is the continuous function with the property that for all $x \in [0,1)$, $\pi \hat r_i(x) =
r_i(x)$ where $\pi : \R \to \R/\Z$ is the usual projection}. Fix $b\in(0,1)$. Consider the vector field $v_F$ on $M_f \times \T^N$ given by
\[
v_F(x,y,w; z_1, \cdots, z_N) =\left\{\begin{array}{ll} (\partial_w; \ b^{-1}\hat r_1(x), \cdots, b^{-1}\hat r_N(x))
& \quad \mbox{for } w \in [0,b)\\
(\partial_w; v_{g_1}, \cdots, v_{g_N})& \quad \mbox{for } w \in [b,1]\ .
\end{array}\right.
\] 

Let $F^t$ be the flow defined by $v_F$. Then $F^t$ is a skew-product 
with the Anosov flow $f^t$ in the base driving the dynamics on $N$ circle fibers 
with coordinates $z_1, \cdots, z_N$. For $w \in [0, b)$,
the action on each $z_i$-fiber is a rotation, the total amount
rotated at the end of $ b$ units of time being $\hat r_i(x)$, while on 
the rest of $M_f$,  fiber dynamics are N-S flows 
independent of the base.

Observe that $v_F$  is discontinuous at the two cross-sections 
$$\Sigma_0:=\{w=0\}  \qquad \mbox{and} \qquad 
\Sigma_{b}:=\{w=b\}  $$
to the flow $F^t$. These discontinuities pose no issues 
in the definition of the flow: a flowline starting from $\Sigma_0$ 
simply follows one vector field until it reaches $\Sigma_{b}$, where it switches to
the other vector field until it returns to $\Sigma_0$. On $\{w \in (0, b)\}$,
$v_F$ is also discontinuous at $\{x=0\}$, though starting from $\Sigma_0$, 
the time-$b$-map $F^{b}$ is well defined and is a $C^2$ 
diffeomorphism from $\Sigma_0$ to $\Sigma_{b}$ {because $r_i$ is a local diffeomorphism}.

\bigskip \noindent
D. {\it The full model: E-I interaction with feedback.} We have described above an E-to-I action 
that takes place on $\{w \in [0,b)\}$. For simplicity we will assume that the
feedback from I to E  takes place only on $\{w \in (b, 1]\}$. To define the latter, 
we modify the {vector field $\partial_w$} in a $z_i$-dependent way:
{Fix $\Phi:\{0,1,...,N\}\rightarrow \R$  an increasing function such that $\Phi(0)=0$ and $\Phi<1$.} Calling $\chi_{(1/2,1)}$ the characteristic function 
of the arc $(1/2,1) \subset \T=\R/\Z$, we define
\begin{equation} \label{eta}
\tilde v_f = \tilde v_{f; z_1, \cdots, z_N} = {{\left[1-\Phi\left(\sum_{i=1}^N\chi_{(1/2,1)}(z_i)\right)\right] }
\partial_w \ :=\  c_{\Phi} (z_1, \cdots, z_N) \partial_w\ .}
\end{equation}

That is to say, the presence of each $z_i$ in the region $(1/2,1)$ of $\T$ causes the Anosov
flow to reduce its speed by {some amount determined by the function $\Phi$. For example one could take $\Phi(n)=\eta n$ with $\eta\in [0,1/N)$, where the inhibitory  effect of multiple $z_i$ being in $(1/2,1)$ is additive.}

The vector field for the full model is then given by 
\[
v_{\bf F}=\left\{\begin{array}{ll} v_{F}& \quad \mbox{for } w \in [0,b)\\
(\tilde v_f; v_{g_1}, \cdots, v_{g_N})& \quad \mbox{for } w \in [b,1]\ .
\end{array}\right.
\] 
The flow associated with the vector field $v_{\bf F}$ is denoted ${\bf F}^t$. 

{Observe that on $\{w \in [b, 1]\}$, the vector field $v_{\bf F}$ is discontinuous at
$\cup_{i=1,\cdots, N} \{z_i = 0, \frac12\}$, the discontinuity set of the function {$c_\Phi$}
in (\ref{eta}). This set  is the union of a finite number of codimension one surfaces.
Because every $g^t_i$ leaves the interval $(\frac12, 1)$ invariant,
flowlines do not cross from one side of these surfaces to the other.}

\bigskip
While $F^t$, the flow without feedback, is a skew product with the Anosov flow $f^t$ in
the base and $z_i$-dynamics in the fibers, this skew-product  structure is not preserved with the
introduction of feedback. { However, because the I-units affect the E-component 
through time changes only, the first return map $H:\Sigma_0\rightarrow \Sigma_0$ of ${\bf F}^t$ 
continues to be a skew-product.} It can be written as $H=H_2\circ H_1$ where
$H_1:\Sigma_0\rightarrow \Sigma_{b}$ is given by
\[
H_1(x,y,0; z_1, \cdots, z_N)=\left(x,y, b;z_1+\hat r_1(x), \cdots, z_N+\hat r_N(x)\right)
\] 
and $H_2:\Sigma_{b} \rightarrow \Sigma_0$ is defined as follows:
\begin{equation}\label{Eq:DefH2}
H_2\left(x,y, b;z_1, \cdots, z_N\right)=(A(x,y),0,G(z_1, \cdots, z_N))
\end{equation}
where 
\begin{eqnarray*}
G(z_1, \cdots, z_N) &=& (g_1^{\tau(z_1,\cdots, z_N)}(z_1), \cdots, g_N^{\tau(z_1,\cdots, z_N)}(z_n))\\
\mbox{and}  \qquad \tau(z_1,\cdots, z_N) &=& 
(1- b){c_\Phi} (z_1, \cdots, z_N)^{-1}.
\end{eqnarray*}
{The formula for $\tau$, the time to go from $\Sigma_{b}$ to $\Sigma_0$,
can be understood as follows: Without feedback, $\tau \equiv 1-b$.
With feedback, it is lengthened by the factor indicated, so that 
$1-b \le \tau \le (1- b)(1-{\Phi(N)})^{-1} :=\tau_{\max}$.  The skew-product structure of $H$ will play a crucial role in our arguments. 

{ To summarize, then, the E-environment acts on the I-units with possibly large rotations
every time the Anosov flow returns to the cross-section $\Sigma_0$. After these rotations, the
I-units collectively determine the subsequent speed of the Anosov flow depending on the positions
to which they are rotated.
The I-units do not interact with one another directly but do so by influencing the time to the next activation,
that is, the time for the Anosov flow to reach $\Sigma_0$. }

\subsection{Technical assumptions}\label{Sec:TecAssump} The following conditions are assumed throughout:

\bigskip \noindent
(a) {\it The fiber flows $g^t_i$.} Recall that for all $i$, 
$g^t_i$ has an attractive fixed point at $z=0$ and a repelling fixed point at
$z=\frac12$. We now assume additionally that for each $i$,
there are numbers {$\delta_i^+, \delta_i^->0$}
and $c \in (0,1)$ such that
if
\[
I_i^+=\left[\frac12-{\delta_i^+},\frac12+{ \delta_i^+}\right] \quad \mbox{and} \quad
I_i^-=\left[0-{ \delta_i^-},0+{ \delta_i^-}\right]\ ,
\]
 then for all $(1- b) \le t \le \tau_{\max}$:
\begin{itemize}
\item[(A1)] $(g_i^t)'|_{I_i^+}> e^{\lambda}$, $(g_i^{t})'|_{I_i^-}\le c$, and
\item[(A2)]  $g^t_i(I_i^+)\supset \T\backslash I_i^-$.
\end{itemize}

\bigskip \noindent
(b) {\it The functions $r_i(x)$.} Recall that the total amount $z_i$ is rotated 
on $\{w \in [0,b)\}$ depends on the $x$-coordinate of the Anosov trajectory 
as explained in Paragraph C. We now impose some additional conditions on 
the rotation function $r_i(x)$. Define
\[
d_i := \mbox{dist}(\partial (I_i^-),\partial g_i^{(1- b)}(I_i^+)).
\]
Then there exist $\varepsilon>0$ small enough 
and $c'>0$ such that
\begin{itemize}
\item[(A3)]  $r_i'>\epsilon^{-1}$  except where $r_i \in (d_i,1-d_i)$,
\item[(A4)]  $r_i'>c'$ everywhere on $\T$.
\end{itemize}
The existence of  $c'$ satisfying (A4) follows form the fact that $r_i$ is a local diffeomorphism.
In Figure \ref{Fig:a1a4} we {illustrate the technical assumptions (A1)-(A4) and give
some examples of $r_i$.}

\begin{figure}
\center
\includegraphics[scale=0.42]{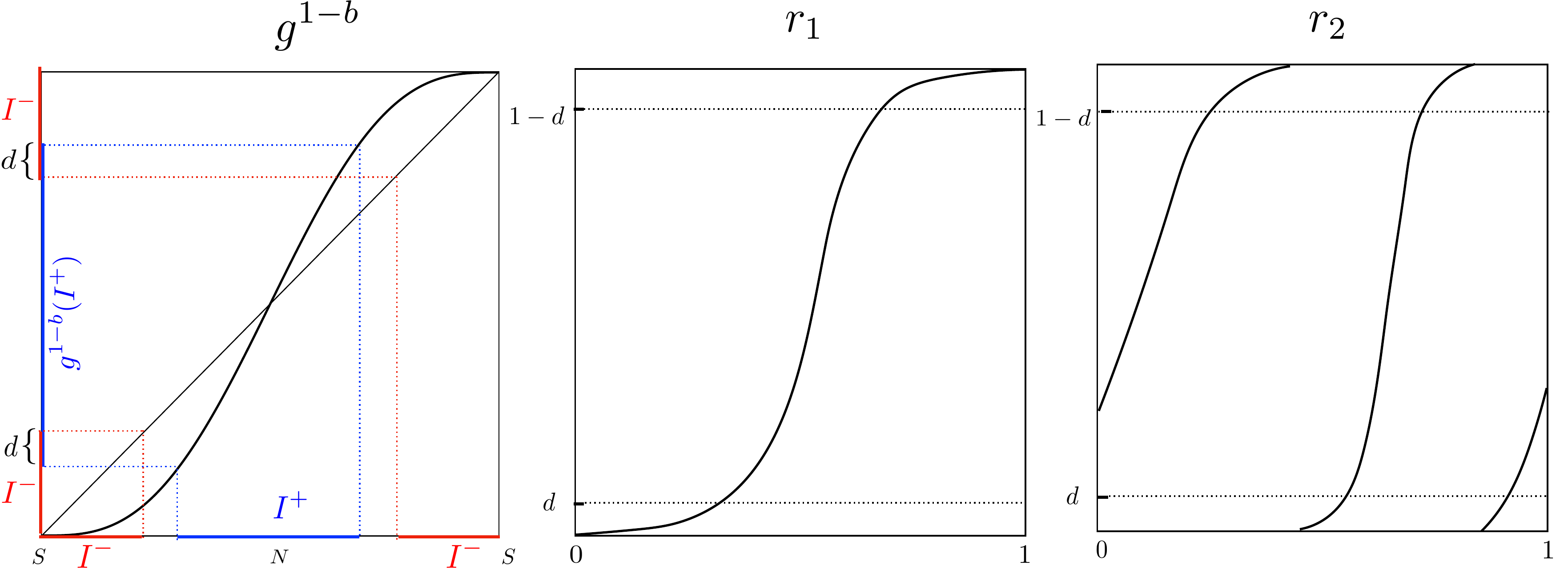}
\caption{The {left panel} above shows an example of the {time-$(1-b)-$map, $g^{(1-b)}$,} for a North-South flow. The diagram shows how $g^{1-b}(I^+)$ should overlap with $I^-$ for (A2) to be satisfied, where $I^-$ and $I^+$ satisfy (A1).  The other two panels show examples of functions $r_i$. {These functions have} the characteristic of being very steep for those $x$ such that $d\le r_i(x) \le 1-d$, and, although possibly very small, its slope is never equal to zero. {Here $r_1$ has degree 1 and $r_2$ has degree two.}}
\label{Fig:a1a4}
\end{figure}
\

{{\paragraph{\bf Example of N-S flow and function $r_i$ satisfying assumptions (A1)-(A4).}
A $2 \times 2$ invertible  matrix $M$ defines an action $\hat g:\R\times \R^2\backslash\{0\}\rightarrow \R^2\backslash\{0\}$  given by $g(t,v)= e^{tM}v$. Let $\mathbb P^1(\R)\cong \T$ be the 1-D real projective space. It is easy to check that $\hat g$ projects to a flow $g^t$ on $ \T$, and
that picking $M$ diagonal with entries $\pm \alpha$ gives rise to a 
North-South flow on $\T$ where $(g^t)' = e^{\pm \alpha t}$ at the two poles.
Given $\lambda$, $b$, and $\tau_{max}$, (A1)-(A2) are satisfied by choosing $\alpha$ 
sufficiently large.

Now fix $\kappa = $ the degree of $r_i$, the N-S flow (so that $d_i$ is fixed), and  $\epsilon\ll \kappa^{-1}$. Pick $\kappa$ disjoint closed arcs on $\T$, $\{[w_{1,\ell},w_{2,\ell}]\}_{\ell=0}^{\kappa-1}$, of length $\epsilon (1-2d_i)$ and assume they are ordered. Define $\tilde r_i:\T\rightarrow \R$ { to be a $C^2$ increasing function with
the following properties: (i) it is affine on $[w_{1,\ell},w_{2,\ell}]$  with constant slope $\epsilon^{-1}$ and $\tilde r_i([w_{1,\ell},w_{2,\ell}])=[d_i+\ell,1-d_i+\ell]$; (ii) on $(w_{2,\ell},w_{1,\ell+1})$, it increases monotonically
from $1-d_i+\ell$ to $d_i+\ell+1$ with $\tilde r_i' > 2d_i$ (independent of $\epsilon>0$). Finally let 
$r_i$ be the projection of $\tilde r_i$  to $\T$.}
}}

%
%
%

\section{Statement of the main results and {an illustrative example}}\label{Sec:Mainresults}
\subsection{Main results}
The model described in Sect. \ref{Sec:PrecModelDesc} is a piecewise smooth flow on a $(3+N)$-dimensional
manifold $M$ where $N \in \Z^+$ is arbitrary. {In finite dimensional systems, observable
events are often equated with positive Lebesgue measure sets, and 
{\it physical measures} are natural invariant measures.}
We begin by recalling the definition of a physical measure.
Below, $m$ denotes the Riemannian measure on $M$.

\begin{definition} Let $\varphi^t$ be a smooth or piecewise smooth flow on $M$.
A $\varphi^t$-invariant Borel probability measure $\mu$ is called a {\bf physical measure} if
there is a Borel set $U \subset M$ with $m(U)>0$ such that for every continuous function
$h: M \to \R$, 
\[
\frac{1}{T} \int_0^T h(\varphi^t(x)) dt \to \int h d\mu \qquad \mbox{as } T \to \infty
\]
for $m$-a.e. $x \in U$.
\end{definition}

If $\varphi^t$ possesses an ergodic probability 
measure $\mu$ absolutely continuous with respect to Lebesgue, then $\mu$ is  
a physical measure by the Birkhoff Ergodic Theorem. But 
the notion of physical measure is meaningful even when $\varphi^t$ is dissipative,
i.e., when {all invariant probability measures
are singular with respect to $m$.}

\bigskip {
We will assert below that under suitable conditions, the flow ${\bf F}^t$ defined in Sect. \ref{Sec:PrecModelDesc} possesses a physical measure. To state these conditions, we
first clarify {the relations among the system's constants}:

\medskip \noindent
(i) chosen freely are $N = \#$ I-components and $\kappa = $ maximum degree of the $r_i$;

\smallskip \noindent
(ii)  $\lambda$,  {the expanding coefficient of the Anosov flow,}, depends on $N$ (see below);

\smallskip \noindent
(iii) the $g^t_i$ are then chosen to satisfy (A1) and (A2) (note (A1) depends on $\lambda$);

\smallskip \noindent
{(iv) $c'$ is chosen {depending on $\kappa$ and $g_i$};} 

\smallskip \noindent
(v) $\epsilon$ depends on everything in (i)-(iii) (see below); and

\smallskip \noindent
(vi) the $r_i$ are then chosen to satisfy (A3)-(A4) (which depends on $\epsilon$).

\smallskip
\begin{theorem}\label{Thm:ExistPhisicMeas} Let ${\bf F}^t$ have the form
 in Sect. \ref{Sec:PrecModelDesc}. There are two functions 
$\lambda(N)$ and $\epsilon(N, \lambda, g^t_i, \kappa)$ such that if 
in addition to satisfying (A1)-(A4), $\lambda$ and $\epsilon$ are chosen so that
$$\lambda \ge \lambda (N) \qquad \mbox{and} \qquad \epsilon \le \epsilon(N, \lambda, g^t_i, \kappa), $$ 
then ${\bf F}^t$ admits a physical measure $\mu$.
\end{theorem}}

\smallskip
For systems with hyperbolic properties, a standard way to produce
physical measures is to construct SRB measures.

\begin{definition}\label{Def:SRBMeasureFlows} Let $\varphi^t$ be a smooth or piecewise smooth flow on $M$.
A $\varphi^t$-invariant Borel probability measure $\mu$ is called an {\bf SRB measure} if

\medskip
(i) $\varphi^t$ has a positive Lyapunov exponent $\mu$-a.e.;

\medskip
(ii) unstable manifolds are defined $\mu$-a.e.

\medskip
(iii) conditional probabilities of $\mu$ on unstable manifolds have densities.
\end{definition}

For definiteness, let us use the term``unstable manifolds" to refer to 
weak unstable manifolds, which includes flowlines, to be distinguished from strong unstable
manifolds, which are one dimension lower. For flows without singularities, 
(ii) follows automatically from (i). In the presence of
discontinuities or singularities, additional conditions are needed to ensure that 
unstable manifolds are defined; {that is why} we have included that as part of the definition.

\smallskip
\begin{theorem}\label{Thm:ExistSRBforflow}  Under the assumptions of Theorem \ref{Thm:ExistPhisicMeas}, the flow ${\bf F}^t$ 
has an ergodic SRB measure $\mu$
with exactly one positive Lyapunov exponent and no zero Lyapunov exponent aside from
the one in the flow direction.
\end{theorem}

\smallskip
Theorem \ref{Thm:ExistPhisicMeas} then follows from Theorem \ref{Thm:ExistSRBforflow} and the absolute continuity of the strong stable
foliation.

\bigskip
The proofs of Theorems  \ref{Thm:ExistPhisicMeas} and \ref{Thm:ExistSRBforflow} are contained in Sects. \ref{Sec:DistPushFor}-\ref{Sec:SRBProp}. For clarity of exposition, 
{\it the proof we present is written for $N=2$.} The case $N=1$ is simpler, and for 
$N \ge 3$, the proof is conceptually identical to that for $N=2$. However, some choices of constants
depend on $N$, the main reason being that the {singularity set} grows in complexity as $N$ increases. We will {identify the dependence on $N$} as we go along.  


{
\subsection{Modeling and simulations}\label{Sec:RelModNumSim}

As remarked in the Introduction, the models considered are not intended to be realistic
models of specific physical or biological systems, but they exhibit a few characteristics 
typical of such systems: Excitatory-inhibitory relations among constituent components, 
the activation of certain processes upon crossing of thresholds,
and the relaxation to equilibrium in the absence of excitatory input --- these properties
appear often in biological systems. 

\begin{figure}[h!]
\center
\includegraphics[scale=0.5]{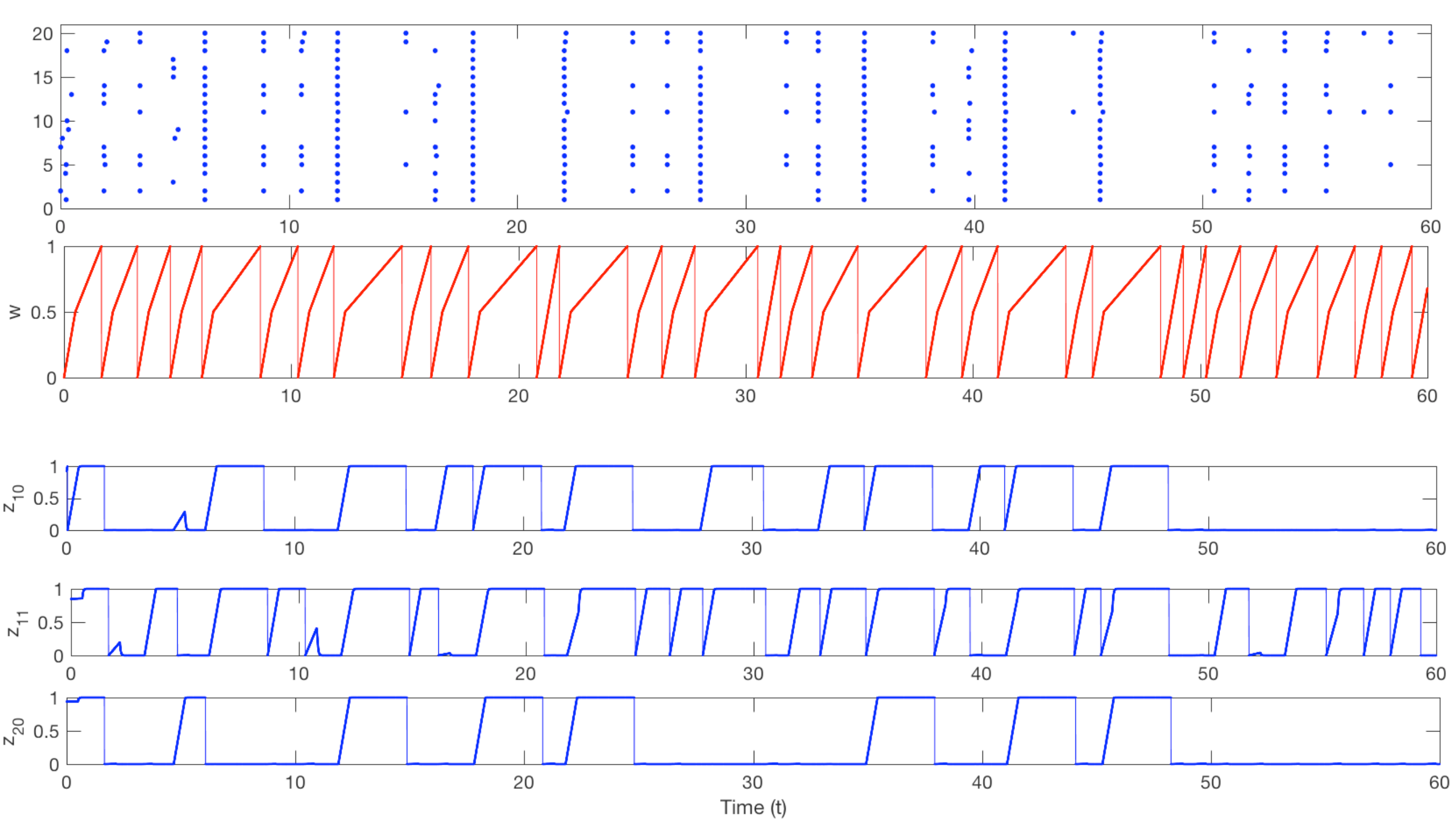}
\caption{Results for numerical simulations of a population of $N=20$ I-units with $v_g(z)=-7 \sin (2\pi z)$, $A=(10,3;3,1)$, $\Phi(n)=n^{-3/10}$, and where the functions $r_i$ are all different having degrees between 2 and 3. {The $x$-axis is time in all the panels.} The first panel shows the Raster plot of the ``activation" events. {The $y$-axis is the I-unit number.} A marker is added at $(t,i)$  when  I-unit number $i$ has coordinate crossing the value 0.5 at time $t$. The second panel shows the $w$ coordinate of the E-environment. The last three panels shows the orbits of I-units number 10 and 11 and 20. {Because their $r_i$ are different, distinct units}   exhibit different activation patterns.}
\label{Fig:RastIncInib}
\end{figure}

%
%
Under (A1)--(A4), the system is hyperbolic (though not uniformly so).
 In many dynamical systems that have come to symbolize chaotic behavior, such as
Anosov diffeomorphisms, billiards or standard maps, the landscape is dominated by 
clearly expanding and contracting directions -- defined everywhere or at least on large 
portions of the phase space. This is not the case in the network models described in 
Section \ref{Sec:Modelandmainres}. In these models, individual I-component are described by simple N-S flows 
with large perturbations mediated by other network components occurring at seemingly
random times. These events inject unpredictability into the system, leading to very rich 
and varied patterns of collective dynamical behaviors 
This is quite typical of 
 high dimensional hyperbolic systems with few unstable directions. Such systems
are not uncommon in real-world settings, and they may be relatively amenable to analysis.

{ In Figure \ref{Fig:RastIncInib} we present an illustrative example of a network model
of the type studied in this paper, to give a sense of its time evolution. Shown are snapshots
of the E-environment and three of the I-components on a randomly chosen time interval. These
snapshots were ``typical" to the degree that they were not chosen with specific characteristics 
in mind; in large dynamical systems, the patterns that can arise are infinitely rich and it is not 
clear what constitutes ``typical" behavior. The character of the dynamics depends on parameters, 
including the choices of $\Phi$, $v_g$, $A$, and especially $\{r_i\}$, the
collection of functions describing the amount the I-component is rotated.

Qualitative properties of some of these dependences can be  analyzed, e.g. since $z_i$ 
is expected to be near $0$ a good fraction of the time (this is what we will prove), 
rotating by $> \frac12$ is more likely to 
cause $z_i$ to cross $0.5$. In this way, one can identify which I-component is likely to 
activate more often. { Moreover components with similar $r_i$  tend to activate simultaneously.} Another observation is that if $\Phi$ is linear, then 
a constraint on the impact of each I-component 
 is that it should scale like $\sim \frac1N$ as system size $N$ 
 tends to infinity, to ensure that $\Phi(\sum_{i=1}^N \chi_{(\frac12, 1)(z_i)})<1$;
 see Equation (1) in Sect. 2.2, Paragraph D. { An interesting question is if, in the limit $N\rightarrow \infty$, the system can be described by a self-consistent operator (as defined e.g. in  \cite{selley2021linear} for systems in discrete time)  and what are the properties of this operator. }

While a complete characterization of the very rich dynamical behaviors of this model is 
likely out of reach, many aspects are amenable to analysis. We will not delve into such
an investigation here, however, as that would take us too far from
our goal of proving the results stated in Section 3.

 }

}

\section{Distribution of pushed-forward mass: a preliminary estimate}\label{Sec:DistPushFor}

{In Sections \ref{Sec:DistPushFor}--\ref{Sec:SRBProp}, unless declared otherwise, $N$ is assumed to be $2$.}

We will, for the most part, be working with the first return map 
$H:\Sigma_0 \circlearrowleft$ of the flow ${\bf F}^t$ defined in Sect. \ref{Sec:PrecModelDesc}. 
Here, $\Sigma_0 \cong \T^2 \times \T^2$ (with Euclidean norm); we use 
$(x,y;z_1, z_2)$ to denote its coordinates. 
Guessing that the bulk of pushed-forward mass is likely to accumulate near $\{z_1=z_2=0\}$, 
the S-poles of the 
fiber flows $(g^t_i, g^t_2)$ in $(z_1,z_2)$-space, so that the Lyapunov exponents of $H$ in
$(z_1,z_2)$-space are likely both negative, we seek to construct an SRB measure 
with 1D unstable manifolds. Our guess is informed by assumptions (A1)-(A4), which 
{were designed to ensure that  mass could not collect around the N-poles.}
 Guided by this intuition, we push forward Lebesgue measure 
on 1D curves roughly aligned with $E^u_A$, hoping to obtain an SRB measure in the limit.  

{The aim of this section is to confirm the hypothesis on mass concentration}.

%

\subsection{Setup and main proposition}\label{Sec:SetupControlGeomCurv}

Let $W^u_A \subset \T^2$ be a piece of unstable manifold for $A: \T^2 \circlearrowleft$,
  and let $\gamma_0:[0,a] \to \Sigma_0$ be a parametrization of $W^u_A \times \{0\}$ by arc length.
For $n \ge 1$, let $\gamma_n : [0, e^{\lambda n}a] \to \Sigma_0$ be defined by
$$
\gamma_n(u) = H^n(\gamma_0(e^{-\lambda n}u))\ ,
$$
so that if $\pi_{xy}$ denotes the projection of $\Sigma_0 \cong \T^2 \times \T^2$ to its first
$\T^2$-component, then $\pi_{xy} \circ \gamma_n$ is a parametrization of $A^n(W^u_A)$ 
by arc length.

\bigskip \noindent
{\bf Action of $H$ in $(u,z_1,z_2)$-space.}
In the analysis to follow, it will be convenient to view the action of $H$ in the
3D space $\R\times \T^2$ with coordinates $(u,z_1,z_2)$, where 
$u$ is arclength parametrization
of $A^n(W^u_A)$. Abusing notation slightly (to avoid excessive numbers of symbols), 
we view $\gamma_n$ as a curve in this space parametrized by 
$$\gamma_n = (u, \pi_{z_1}(\gamma_n(u)), 
\pi_{z_2}(\gamma_n(u)), \qquad  u \in [0, e^{\lambda n}a], $$
where $\pi_{z_i}$ has the obvious meaning,
 and study the evolution of $\gamma_n$ in this setting under the iteration of $H$. 
 
 To define the action of $H$, we need to express the amount of rotation as functions of $u$. 
We introduce a $\gamma_0$-dependent family of functions
$$\iota_n:[0, e^{\lambda n}a ]\rightarrow \T \ , \quad \quad
\iota_n(u) = \pi_x(\gamma_n(u))\ .
$$
Then $\iota_n(u)$ has the form 
$\iota_n(u)=\beta u+{\pi_x(\gamma_n(0))}$ where $\beta$ is the angle between the unstable direction 
$E^u_A$ and the $x$ direction and {$\pi_x(\gamma_n(0))$} is a  constant depending on $\gamma_n$.

The action of $H$ on $(u,z_1, z_2)$-space can now be written as
$\Phi:=\Phi_3\circ \Phi_2\circ \Phi_1$ where
\begin{align*}
\Phi_1&: \,(u;z_1,z_2)\mapsto (u; z_1+ r_1(\iota_n(u)), z_2+r_2(\iota_n(u)))\\
\Phi_2&:\,(u;z_1,z_2)\mapsto (u; g_1^{\tau_1(z_1,z_2)}(z_1), g_2^{\tau_2(z_1,z_2)}(z_2))\\
\Phi_3&:\,(u;z_1,z_2)\mapsto (e^\lambda u; z_1, z_2)
\end{align*}
This is the setup we will be working with. 

Observe that in this setup, the curves of interest are all graphs of functions in the variable $u$: The curve $\gamma_0$
is the graph of the function $\theta_0(u) \equiv (0,0)$. Suppose we show inductively that
$$\gamma_n = \mbox{graph}(\theta_n) \quad \mbox{for some} \quad
\theta_n=(\theta_n^1,\theta_n^2) :[0,e^{\lambda n}a ]\rightarrow \T^2 .
$$ 
Then $\Phi_1$ transforms the graph of $\theta_n$ to the graph of another function 
$\psi_n = (\psi_n^1,\psi_n^2)$, i.e. $\Phi_1(\mbox{graph}(\theta_n)) = \mbox{graph}(\psi_n)$,
where
$$
\psi_n(u) = \theta_n(u) + (r_1(\iota_n(u)), r_2(\iota_n(u))\ .
$$
Likewise, $\Phi_2$ transforms the graph of $\psi_n$ to the graph of another function 
$\zeta_n = (\zeta_n^1,\zeta_n^2)$ with
$$
\zeta_n(u) = (g_1^{\tau(\psi_n(u))}(\psi^1_n(u)), g_2^{\tau(\psi_n(u))}(\psi^2_n(u))),
$$
and $\Phi_3$ transforms the graph of $\zeta_n$ to the graph of $\theta_{n+1}$ with 
$$\theta_{n+1}(u) = \zeta_n(e^{-\lambda}u)\ .$$

As discussed at the beginning of this section, we are interested in where mass accumulates
in $(z_1, z_2)$-space starting from uniform measure on $\gamma_0$ and transporting mass
forward by $H$. Below we translate this statement into one about $\psi_n$. The symbols
$I^-_i$ and $\varepsilon$ refer to those in Sect. \ref{Sec:TecAssump}.

\begin{proposition}\label{Prop:ConcNearToSink} Assume $e^\lambda > 3$ and $\varepsilon>0$ is sufficiently small.
Then there exists $C_\#>0$ ({ dependent on $\kappa$, $c'$, $a$, $\beta$, and the N-S flows $g_i$}) such that  for all $n \in \N$ and  $i=1,2$,  
\[
\frac{1}{e^{\lambda n}a} \cdot m\{u \in [0, e^{\lambda n}a] : \psi^i_n(u) \in (I^-_i)^c\} 
< C_\# \varepsilon\ 
\]
where $m$ denotes Lebesgue measure on $\R$.
\end{proposition}

This is the main result of this section.
Had $\psi^i_n$ been smooth and monotonically increasing, a natural way to 
estimate $m\{\psi^i_n(u) \in (I^-_i)^c\}$ for each $i$
would be to examine the derivative of $\psi^i_n$ and to 
count how many times its graph wraps 
around $\T$ as $u$ varies over the interval $[0, e^{\lambda n}a]$. 
But as we will show, the functions $\psi^i_n$
are neither continuous nor monotonic due to discontinuities in our model. 
In Sect. \ref{Sec:GeomItCurv}, we will investigate the extent to which they fail to be continuous
and monotonic. The proof of Proposition \ref{Prop:ConcNearToSink} is given in Sect. \ref{Sec:ProofMassConc}.

\subsection{Geometry of iterated curves}\label{Sec:GeomItCurv}

Starting from 
$\theta_0^1 \equiv 0$, we see from the definition of 
$\Phi:=\Phi_3\circ \Phi_2\circ \Phi_1$ that $\Phi_1$ and $\Phi_3$ will take a smooth
piece of graph that is monotonically increasing to another such piece of graph, and new
singularities can only be created as we pass from $\psi_n$ to $\zeta_n$, so we focus
below on the action of $\Phi_2$.

The following language will be used to facilitate the discussion: 
We say a function $h: \R \to \R$ is
``monotonic at $x$" if there is a neighborhood of $x$ on which $h$ is monotonic, 
$C^1$ at $x$ if there is a neighborhood on which $h$ is $C^1$, and
``$C^{1*}$ at $x$" if there is a neighborhood on which $h$ is $C^1$ except possibly at $x$.
We call $x$ a``singularity" of $h$ if it is not $C^1$ at $x$, and use $S(h)$ to denote the
set of singularities.

\begin{lemma}\label{Lem:Propofpsiandlift} The following hold for all $n \in \N$ and $i=1,2$:
\begin{itemize}
\item[(i)] The set $S(\psi^i_n)$ is finite, and 
$\exists \ c_1 >0$ { depending only on $c'$  and the N-S flow $g_i$ (independent of $n$)}
such that $(\psi^i_n)'(u) \ge c_1$ for all 
$u \not \in  S(\psi^i_n)$.
\end{itemize}
Suppose $\psi^i_n$, $i=1,2$, is $C^1$ (resp. $C^{1*}$) and monotonic at $\bar u$. 
\begin{itemize}
\item[(ii)] If $\psi^i_n(\bar u) \ne 0, \frac12$  for $i=1,2$, then
$\zeta^i_n$ is $C^1$ (resp. $C^{1*}$) and monotonic at $\bar u$.
\item[(iii)] If $\psi^1_n(\bar u) = 0$ or $\frac12$, then 
$\zeta^1_n$ is $C^{1*}$ and monotonic at $\bar u$.
\item[(iv)] If $\psi^1_n(\bar u) \ne 0, \frac12$ and $\psi^2_n(\bar u) = 0$ or $\frac12$, then 
$\zeta^1_n$ is discontinuous and not necessarily monotonic at $\bar u$.
\end{itemize}
Statements (iii) and (iv) hold with $i=1,2$ interchanged.
\end{lemma}

Once Item (i)  is proved, it will follow that for 
$\bar u \not \in S(\psi^i_n)$, if $\zeta^i_n$ is continuous at $\bar u$, then it is automatically monotonic;
if $\zeta^i_n$ is discontinuous at $\bar u$, then it either jumps up or down, and {a violation of monotonicity}
is created if and only if $\zeta^i_n$ jumps down.

\begin{proof} The dynamics on $\T^2=(z_1, z_2)$-space are generated by a single $C^2$
vector field $(v_{g_1}, v_{g_2})$, 
but $\T^2$ is divided into 4 quadrants by the lines $z_i=0, \frac12, \ i=1,2$, and 
singularities of $\Phi_2$ arise from the fact that the flow times $\tau$ in $g^\tau_i$ are different 
on these 4 quadrants. 

The assertion in (ii) follows immediately. Item (i) is proved inductively, using (ii), (A4), 
which guarantees a lower bound on $r_i'$, 
and the fact that $\Phi_2$ and $\Phi_3$ can only decrease the derivative of $\psi^i_n$ 
by a finite amount {}.

For (iii) and (iv), it is useful to picture a small piece of curve $u \mapsto \psi^i_n(u)$ in $\T^2$
defined for $u$ in a neighborhood of $\bar u$, 
and see how it is transformed to $u \mapsto \zeta^i_n(u)$ by applying the appropriate flow-maps.
In the scenario of (iii), different flow-times are applied to $u<\bar u$ and $u > \bar u$, but no
discontinuity is created because $0$ and $\frac12$ are fixed points of
the flow $g^t_1$. A jump discontinuity in $\zeta_n^1$ is created in the scenario in (iv). This jump 
can be up or down depending on the location of $\psi_n(\bar u)$: 
For example, if $\psi^2_n(\bar u) =\frac12$ and
$\psi^1_n (\bar u) \in (0,\frac12)$, then $\tau(\psi_n(u))$ for $u<\bar u$ is smaller than
 $\tau(\psi_n(u))$ for $u > \bar u$, and since $g^1_t$ is decreasing at $\bar u$, 
 the jump is downwards, i.e., $\zeta_n^1$ loses monotonicity at $\bar u$. Similar reasoning
will show that the jump is upwards (hence $\zeta^1_n$ retains its monotonicity) if 
we take $\psi^2_n(\bar u) =0$ and $\psi^1_n (\bar u)$ as before.
\end{proof}

As it turns out, isolated points of nondifferentiability are of no concern to us for as long
as the function is continuous at those points, i.e., the singularities created in scenario (iii)
are harmless. Our concern is (iv). Note also that discontinuities of $\zeta^1_n$ are 
determined not by $\psi^1_n$ but by $\psi^2_n$.

\medskip 
As downward jumps are unavoidable, we try to control the distances jumped.
Let $h : [0,b] \to \T$ be a piecewise
continuous, monotonically increasing function, i.e., there exist $0=d_0 < d_1 < \cdots < d_k = b$ such that 
$h$ is continuous and monotonically increasing on each interval $(d_{j-1}, d_j), j=1, \cdots, k$, 
and it has jump discontinuities at $d_1, \cdots, d_{k-1}$. Notice that no conditions are imposed on 
the relation between $h^-(d_j)$ and $h^+(d_j)$ where $h^\pm(d_j)$ denote the left
and right limits of $h$ at $d_j$. Given such an $h$, we let 
$\Gamma h : [0,b] \to \R$ be the function with the following properties:
\begin{itemize}
\item[(i)] $(\Gamma h)(u) = h(u)$ mod 1 for all $u \in [0,b]$;
\item[(ii)] $\Gamma h$ is monotonically increasing on $[0,b]$, i.e. $u_1 \le u_2 \implies
(\Gamma h)(u_1) \le (\Gamma h)(u_2)$ for all $u_1, u_2 \in [0,b]$;
\item[(iii)] $(\Gamma h)^-(d_j) < (\Gamma h)^+(d_j) < (\Gamma h)^-(d_j)+1$;
\item[(iv)] $(\Gamma h) (0) \in [0,1)$.
\end{itemize}
That is to say, $\Gamma h$ changes $h$ into a monotonically increasing function 
with jump sizes $<1$ while preserving the values of $h$ mod $1$. 
It is obvious that $\Gamma h$ is uniquely defined and has jump discontinuities at the
same locations as $h$.

For $i=1,2$ and all $n$, let
$$
\hat \theta^i_n = \Gamma(\theta^i_n), \qquad \hat \psi^i_n = \Gamma(\psi^i_n), \qquad
\hat \zeta^i_n = \Gamma(\zeta^i_n)\ .
$$
We let $\mc R(\hat \theta^i_n)$ denote the length of the shortest interval containing 
the image of $\hat \theta^i_n([0, e^{\lambda n}a])$; $\mc R(\hat \psi^i_n)$ and 
$\mc R(\hat \zeta^i_n)$ are defined analogously. We are interested especially in
$\mc R(\hat \psi^i_n)$, which as we will see is an overcount for the number of times
$\psi^i_n$ wraps around $\T$ (see the last paragraph of Sect. \ref{Sec:SetupControlGeomCurv}).

\begin{proposition}\label{Prop:RangeLiftsFundDomains}
Assume that $e^\lambda>3$. Then there is $c_2>0$ ({depending on $\kappa$, $\beta$, $a$, and $\lambda$}) such that for every $n\in\N$ and $i\in\{1,2\}$,
\begin{align}\label{Eq:UppBndRangepsi}
\mc R(\hat \psi^i_n)\le c_2e^{\lambda n}.
\end{align}
\end{proposition}

\begin{proof} 
Let $\kappa$ be such that the degree {of $r_i$ is $\le \kappa$} for every $i$.

We claim that for $i=1,2$ and for $n \ge 1$, 
\begin{eqnarray*}
\mc R(\hat \theta^i_n) & = & \mc R(\hat \zeta^i_{n-1}), \qquad  
\mc R(\hat \psi^i_n) \ \le \ \mc R(\hat \theta^i_n) +  \kappa e^{\lambda n}a\beta + \kappa, \\
\mc R(\hat \zeta^i_{n})& \le  &\mc R(\hat \psi^i_n) + 1 + X_n^i
\end{eqnarray*}
where $X_n^i$ is the number of new discontinuities created as we go 
from $\psi^i_n$ to $\zeta^i_n$.
The first equality is clear, as $\theta^i_n$ is simply a rescaling of $\zeta^i_n$ in $u$.
The quantity $\kappa e^{\lambda n}a\beta$ comes from the fact that the rotation functions $r_i$ 
have degree bounded by $\kappa$, and $\beta$ is the modulus of the cosine of the angle between $W^u_A$ and the $x$-axis.
The ``$+\kappa$" in $\mc R(\hat \psi^i_n)$ and ``$+1$" in  $\mc R(\hat \zeta^i_n)$
are end point corrections: the amount rotated can be arbitrarily close to  $\kappa\lceil e^{\lambda n}a\beta\rceil$, i.e. the smallest integer $\ge e^{\lambda n}a\beta$  times $\kappa$,
while $g^t_i$ can increase the length of an arc by $\frac12$ at either end.

We claim that $X^1_n \le 2(\mc R(\hat \psi^2_n) +1)$.
This is because by Lemma \ref{Lem:Propofpsiandlift} (iv), discontinuities
are created only when $\psi^2_n$ crosses $z_2=0,\frac12$, and 
the number of times that can happen is $ \le 2(\mc R(\hat \psi^2_n) +1)$. 
Similarly, $X^2_n \le 2(\mc R(\hat \psi^1_n) +1)$.

Letting $\mc R(\hat \psi_n) = \max_{i=1,2} \mc R(\hat \psi^i_n)$ and similarly for
$\mc R(\hat \zeta_n)$, we deduce from the relations above that
\[
\mc R(\hat \zeta_{n}) \le 3 \mc R(\hat \psi_n) + 3 \le 3 (\mc R(\hat \zeta_{n-1}) +\kappa e^{\lambda n}a\beta
 + \kappa) + 3\ .
\]
Applying this relation recursively, we obtain
\begin{eqnarray*}
\mc R(\hat \zeta_{n}) & \le & 3\mc R(\hat \zeta_{n-1}) + 3 \kappa e^{\lambda n}a \beta+  3+3\kappa\\
& \le & 3^2\mc R(\hat \zeta_{n-2}) + \kappa a \beta (3 e^{\lambda n} + 3^2 e^{\lambda (n-1)})  +(3+1) (3+3\kappa) \\
& \le & \quad \cdots \\
& \le & 3^n \mc R(\hat \zeta_0) + 3 \kappa a\beta \sum_{i=0}^n 3^i e^{\lambda (n-i)} + ( 3+3\kappa) \sum_{i=0}^n 3^i\ .
\end{eqnarray*}
The first and third terms are $\mc O(3^n)$ as $n \to \infty$. 
Since $e^{\lambda} > 3$ is assumed, the middle term
is $\mc O(e^{\lambda n})$. It follows that $\mc R(\hat \zeta_{n}) = \mc O(e^{\lambda n})$,
from which the asserted statement follows.
\end{proof}

\begin{remark}
Notice that {while the condition $e^{\lambda} > 3$ in Proposition \ref{Prop:RangeLiftsFundDomains}
is sufficient for $N=2$, more stringent lower bounds for $\lambda$ are needed as $N$, 
the number of  I-components,
increases. This is because} going from $\psi^1_n$ to $\zeta^1_n$, jump discontinuities are created 
when $\psi^i_n$ crosses $0, \frac12$ for all $i \ne 1$. Thus more jump discontinuities are
created at each step for larger $N$, and a larger lower bound for $\lambda$ is required to beat 
the growth in number of singularities.
\end{remark}

\subsection{Proof of mass concentration}\label{Sec:ProofMassConc}

In addition to the results from Sect. \ref{Sec:GeomItCurv}, the proof below will rely heavily on Assumptions
(A1)-(A4) in Sect. \ref{Sec:TecAssump}, and the notation will be as in that section.

\begin{proof}[Proof of Proposition \ref{Prop:ConcNearToSink}] The proof consists of two steps.

\medskip
{ To ensure that mass is concentrated near S-pole,} 
our first step is to show that for $i=1,2$ and all $n \in \N$,
\begin{equation} \label{slope}
\mbox{if} \ u \not \in S(\psi^i_n) \mbox{ is such that } \psi^i_n(u) \in (I_i^-)^c, \ \mbox{ then } 
(\psi_n^i)'(u)\ge \beta\epsilon^{-1}\ .
\end{equation}
We fix $i$, and prove the claim by induction on $n$. 

{ The idea is to look { at  $u$ such that $\psi^i_n(u)$ is near N-pole and compare it with  
{ $\psi^i_{n-1}(e^{-\lambda} u)$}. 
If { $\psi^i_{n-1}(e^{-\lambda} u)$} was 
near N-pole}, then { $(\psi^i_{n-1})'(e^{-\lambda} u)$} 
was  large { by induction}, and there are no mechanisms
{that can decrease this derivative by too much going to $\psi_n^i(u)$}; if it was far from N-pole, then
it must have experienced a large rotation to get there at step $n$, and (A3) ensures that $r_i$ creates
a very steep slope for such rotations.

More precisely, assume that (\ref{slope}) is true for $n-1$, and let 
$u$ be as in (\ref{slope}).}
Then
\[
\psi_n^i(u)=g_i^{\tau(\psi_{n-1}(e^{-\lambda}u))}(\psi_{n-1}^i(e^{-\lambda}u))+\hat r_i(\iota_n(u))\ ,
\]
and since  $u \mapsto \tau(\psi_{n-1}(e^{-\lambda}u))$ is locally constant, we have
\[
(\psi_n^i)'(u) = (g_i^{\tau(\psi_{n-1}(e^{-\lambda}u))})'(\psi_{n-1}^i(e^{-\lambda}u)) \cdot
 (\psi_{n-1}^i)'(e^{-\lambda}u) \cdot e^{-\lambda}
+(\hat r_i\circ \iota_n)'(u)\ .
\]
We consider separately the following two cases.

\medskip \noindent
{\it Case (a):} $\psi_{n-1}^i(e^{-\lambda}u)\in I_i^{+}$. Here we claim to have
\[
(\psi_n^i)'(u) \ge e^{\lambda} \cdot \beta \epsilon^{-1}\cdot e^{-\lambda} = \beta\epsilon^{-1}\ .
\]
That $(g_i^{\tau(\cdots)})'(\psi_{n-1}^i(e^{-\lambda}u)) > e^\lambda$ follows from (A1) 
together with $\psi_{n-1}^i(e^{-\lambda}u)\in I_i^{+}$. That $(\psi_{n-1}^i)'(e^{-\lambda}u)
\ge \beta \varepsilon^{-1}$ follows from the induction hypothesis, and the $\hat r_i$-term
can be dropped because it is $>0$.

\medskip \noindent
{\it Case (b): $\psi_{n-1}^i(e^{-\lambda}u)\in (I_i^{+})^c$.} It follows from (A2) that 
$g_i^{\tau(\psi_{n-1}(e^{-\lambda}u))}(\psi_{n-1}^i(e^{-\lambda}u))$ is at least 
distance $d_i$ away from $\partial I_i^{-}$. In order for $\psi_n^i(u)$ to belong to $(I_i^{-})^c$,  
$\hat r_i(\iota_n(u))$ must be $>d_i$, so $r_i'(\iota_n(u))>\epsilon^{-1}$ by (A3). Since
$(\psi_{n-1}^i)' >0$, we have
\[
(\psi_n^i)'(u)\ge (\hat r_i\circ \iota_n)'(u)=\hat r_i'( \iota_n(u))\iota_n'(u)>\beta\epsilon^{-1}.
\]

\medskip
%

{ Our second step is to { use} the derivative information above { to bound} the number of
times the graph of $\psi^i_n$ meets the interval $(I_i^-)^c$. Here we have taken care
of the fact that $\psi^i_n$ is not monotonic by replacing it with $\hat \psi^i_n$ (see
Sect. 4.2).}
Define 
\[G^i_n = \{u\in [0, e^{\lambda n}] : \hat \psi^i_n(u) \in (I_i^-)^c + \Z\}\ .
\]
Thinking of the gaps at jumps as pieces of graph with infinitely large slopes, we see
that the number of components of $(I_i^-)^c + \Z$ reached by the range of $\hat \psi^i_n$ is
$\le \mc R(\hat \psi_n^i)+2$, and we have shown that 
$\mc R(\hat \psi_n^i) < c_2 e^{\lambda n}$ (Proposition \ref{Prop:RangeLiftsFundDomains}). If $J$ is a component of $(I_i^-)^c + \Z$,
then $|(\hat \psi^i_n)^{-1}(J)| \le |(I^{-})^c|\beta^{-1} \epsilon$  from Step 1. Altogether, we have
\[
\frac{1}{e^{\lambda n}a} \ m(G^i_n) < \frac{1}{e^{\lambda n}a} \cdot c_2 e^{\lambda n} \cdot
|({ I_i^{-}})^c|\beta^{-1} \epsilon < { C_\#}\varepsilon
\]
for some ${C_\#}$, completing the proof.
\end{proof}

\section{A candidate SRB measure for the return map}\label{Sec:CandidateSRB}

Assumptions (A1)--(A4) are in effect throughout.
We continue to develop the ideas outlined at the beginning of Sect. \ref{Sec:DistPushFor}, namely to construct
a candidate SRB measure for the first return map $H:\Sigma_0 \circlearrowleft$ of the flow 
${\bf F}^t$ by pushing forward Lebesgue measure on curves with a component in 
the unstable direction. Let $\gamma_0$ be as defined in Sect. \ref{Sec:SetupControlGeomCurv},
and let $m_{\gamma_0}$ be Lebesgue measure on $\gamma_0$. Assuming $a=1$ 
so $m_{\gamma_0}$ is a probability measure, we let 
\[
\mu_n := \frac{1}{n} \sum_{i=0}^{n-1} H^n_*(m_{\gamma_0})\ .
\]

\subsection{Limit points of $\mu_n$ and relation to singularities}\label{Sec:LimiRelSing}

The main obstacle to concluding any limit point of $\mu_n$ is an $H$-invariant probability measure
with desirable properties is the presence of singularities, so that is what we will focus on.

Recall that $H=H_2 \circ H_1$ (Sect. \ref{Sec:PrecModelDesc}), and that $H_1: \Sigma_0 \to \Sigma_{b}$
is a diffeomorphism whereas $H_2: \Sigma_{b} \to \Sigma_0$ is piecewise smooth
with discontinuities. Let $\mc S(\cdot)$ denote the singularity set of a map. 
Then
\[
\mc S(H_2) = \sigma_{b}\times S
\]
where $S:=S_1\cup S_2$, $S_i:=\{z_i=0, \frac12\}{ \subset{\T^2}}$, and 
$\sigma_{b} := {\T^2\times\{b\}\subset  M_{f}}$.
Because $\mc S(H_2)$ has a simpler geometry than 
$\mc S(H) = H_1^{-1}\mc S(H_2)$, and $\mc S(\bar H) = \mc S(H_2)$ for
$\bar H = H_1 \circ H_2$, it is simpler to work 
with $\bar H$. In the rest of Sect. \ref{Sec:LimiRelSing},
we will consider 
\begin{equation}\label{Eq:SeqMeasures}
\bar \mu_n:=\frac{1}{n}\sum_{i=0}^{n-1}\bar H_*^{i} ({H_1}_*m_{\gamma_0})\ , \quad
n=1,2,\cdots.
\end{equation}
 The $(H_1^{-1})_*$-image of any weak$^*$-accumulation point of 
$\{\bar \mu_n\}$ is clearly an accumulation point of $\{\mu_n\}$.

\medskip
We show below that  weak$^*$-accumulation points of $\{\bar \mu_n\}$, which
 exist in the space of all Borel probability measures by compactness, are invariant measures
of $\bar H$ with controlled properties near its singularity set.
 For $X\subset \T^m$ and $\epsilon>0$, we denote by $X_\epsilon$ the 
$\epsilon$-neighborhood of $X$.

\begin{lemma}\label{Prop:MeasureAroundSing} There exists $c_3>0$ ({ with the same dependencies of $c_1$ and $c_2$}) such that 
if $\{\bar \mu_n\}_{n\in\N}$ is as defined in \eqref{Eq:SeqMeasures}, then
for all ${\xi}>0$,
\begin{equation}\label{Eq:PropMun}
\bar \mu_n ({\sigma_{b}}\times S_{{\xi}})\le c_3 {\xi}\ .
\end{equation}
\end{lemma}

\begin{proof}
We estimate $\bar \mu_n( \sigma_{b}\times(S_i)_{{\xi}})$ for $i=1$. The case $i=2$ is analogous.

Recalling the definition of $\psi_n$, one can see that {
\[
 \bar H_*^i ({H_1}_*m_{\gamma_0})\left(\sigma_{b}\times (S_1)_{{\xi}}\right)=\frac{1}{e^{\lambda i}}m\left(({\hat \psi_i^1})^{-1}(S_1)_{{\xi}}\right).
\] 
}
By point i) of Lemma \ref{Lem:Propofpsiandlift} $(\psi_i^1)'\ge c_1>0$ where $\psi_i^1$ is differentiable, therefore the preimage of each copy of $ (S_1)_{{\xi}}$ in one fundamental domain on the lift has measure at most $2{\xi} c_1^{-1}e^{-\lambda i}$.
By Proposition \ref{Prop:RangeLiftsFundDomains}, ${\hat \psi_i^1}$ spans at most $c_2e^{\lambda i}$ fundamental domains. The result is proved with $c_3 = 4c_2 c_1^{-1}$.
\end{proof}

\begin{proposition}\label{Prop:InvMeasBarH} Under Assumptions (A1)--(A4), 
any accumulation point $\bar \mu$ of the sequence in  \eqref{Eq:SeqMeasures} is 
an invariant probability measure for $\bar H$, and it satisfies 
\begin{equation}\label{Eq:ErgCompSingMeas}
\bar \mu\left({\sigma_{b}}\times S_{{\xi}}\right)\le c_3{\xi}
\end{equation}
for every ${\xi}>0$, where $c_3$ is as in Lemma \ref{Prop:MeasureAroundSing}.
\end{proposition}

\begin{proof} Let $\mc M$ be the set of all Borel probability measures on $\Sigma_{b}$,
and let 
\[
\mc M_{S}:=\left\{\nu \in \mc M :\, \nu\left(\sigma_{b}\times S_{{\xi}}\right)\le c_3{{\xi}}\,\,\,\forall {\xi}>0\right\}.
\]
Then by Lemma \ref{Prop:MeasureAroundSing}, $\bar \mu_n \in \mc M_S$. To use the standard Krylov-Bogolyubov argument 
to prove that any accumulation point of $\{\bar \mu_n\}$ is fixed by $\bar H_*$, it suffices to show that
$\bar H_*$ acts continuously on $\mc M_S$.
Let $\nu_n \in \mc M_S$ be such that $\nu_n$ converges to $\nu$ in the 
weak$^*$ topology, and fix a continuous function $\phi:\Sigma_{b}\rightarrow \R$.
We will show $\int \varphi d(\bar H_*\nu_n) \to \int \varphi d(\bar H_*\nu)$.

Given an arbitrarily small ${\delta}>0$, 
let ${\xi}>0$ be small enough  so that  
  \begin{equation}\label{Eq:ChoiceOfEps}
 c_3{\xi}\sup_{p\in\Sigma_{b}}|\phi(p)|\le \frac{\delta}{2}.
 \end{equation}
Let $U_1=({\sigma_{b}}\times S_{{\xi}/2})^c$,  $U_2:=({\sigma_{b}}\times S_{{\xi}})^c$,  $U_1':=\bar H U_1$ and $U_2':=\bar H U_2$. Then $U_1'$ and $U_2'$ are 
closed sets with $U_2' \subset \mbox{int}(U_1')$. Let ${\rho}:\Sigma_{b}\rightarrow [0,1]$ 
be a continuous function with ${\rho}|_{U_2'}=1$ and ${\rho}|_{{U_1'}^c}=0$. 
Since the support of ${\rho}$ is contained in $U_1'$ and that of $(1-{\rho})$ is contained
in $(U'_2)^c$, we have, by the injectivity of $\bar H$,
\begin{equation} \label{supp}
\mbox{supp}({\rho}\circ \bar H) \subset U_1, \qquad \mbox{supp}((1-{\rho})\circ \bar H) 
\subset U_2^c={\sigma_{b}}\times S_{{\xi}}\ .
\end{equation} 
Writing
\[
 \int_{\Sigma_{b}} \phi d\bar H_*\nu_n = 
\int_{\Sigma_{b}}(((1-{\rho})\phi)\circ \bar H) d\nu_n
 +\int_{\Sigma_{b}}(({\rho}\phi)\circ \bar H) d \nu_n,
 \]
we have that the first integral on the right is $< \frac12 {\delta}$ by (\ref{Eq:ChoiceOfEps}) and (\ref{supp}).
Since the support of ${\rho} \circ  \bar H$ is bounded away from ${\sigma_{b}}\times S$,
the second integral converges to $\int ({\rho} \varphi) \circ \bar H d\nu$ as $n \to \infty$.
It follows that 
\[
\left| \int \phi d(\bar H_*\nu_n) - \int \phi d(\bar H_*\nu) \right|\le {\delta}
 \]
for all large $n$, proving the convergence claimed.

To complete the proof, observe that $\mc M_S$ is closed and therefore every 
accumulation point belongs to $\mc M_S$.
\end{proof}

\medskip
The results above imply the following: Any limit point $\mu$ of $\mu_n$ where $\mu_n$ is as defined
at the beginning of Sect. \ref{Sec:CandidateSRB} is an invariant probability measure of $H$. From here on
let $\mc S = \mc S(H)$ denote the singularity set of $H$. Then it follows from Proposition \ref{Prop:InvMeasBarH}
that $\mu(\mc S_{\varepsilon_1}) <$ const $\varepsilon_1$
for every $\varepsilon_1>0$.

\subsection{Lyapunov exponents of $(H, \mu)$}\label{Sec:LEHmu}

In this section we study the Lyapunov exponents (LE) of $(H, \mu)$ where 
$H:\Sigma_0 \circlearrowleft$ is the first return map and $\mu$ is a limit
point of the sequence of measures $\mu_n$. Recall that $H$ has a skew-product structure:
the base is $A: \T^2 \circlearrowleft$, and fiber variables are $(z_1, z_2) \in \T^2$. 
Let $E_{z_i} = \langle \partial_{z_i} \rangle$ denote the 1-dimensional space spanned 
by $\partial_{z_i}$, and let $E_{z_1z_2} = \langle \partial_{z_1}, \partial_{z_2} \rangle$ be the space spanned by $\partial_{z_1}$ and $\partial_{z_2}$.
We {claim} that for $p \in \Sigma_0 \setminus \mc S$, $DH_p(E_{z_i}) = E_{z_i}$ for 
$i=1,2$: Writing
\[
DH_p|_{E_{z_1z_2}} = D(H_2)_{H_1(p)}|_{E_{z_1z_2}} \circ D(H_1)_p|_{E_{z_1z_2}}\ ,
\]
we have $D(H_1)_p|_{E_{z_1z_2}}= \mbox{Id}$ because {on $\{\pi_{xy}(p)\} \times \T^2$,}
 $H_1$ is a rigid translation, {and writing $q=H_1(p)$,}
\[
D(H_2)_q|_{E_{z_1z_2}}\binom{w_1}{w_2}=\binom{D(g_1^{\tau})_{\pi_{z_1}q} w_1}{D(g_2^{\tau})_{\pi_{z_2}q} w_2}\ ,
\]
where $\tau = \tau(\pi_{z_1}q, \pi_{z_2}q)$ is a locally constant function. This proves 
{the claim.}

{Recalling that $\mu$ gives zero measure to the singularity set, so LEs are
defined $\mu$-a.e.,} we let $\lambda_{z_i}(p)$ denote the LE at $p$ in the $E_{z_i}$-direction.
Below $\varepsilon>0$ is the number in Sect \ref{Sec:PrecModelDesc}, Technical assumptions (b).

\begin{lemma} \label{Lem:NegExp}Assuming that $\varepsilon>0$ is small enough,
there is an $H$-invariant {measure $\hat \mu = \mu|_B$, the restriction of
$\mu$ to a Borel subset $B$,
such that for $\hat \mu$-a.e. $p$,}
$\lambda_{z_i}(p)<0$ for $i=1,2$.
\end{lemma}

\begin{proof} It follows from the assumptions in Sect \ref{Sec:TecAssump} 
that there exist $C>1$ and $c < 1$
 independent of $\varepsilon$ such that 
for $i=1,2$,

\medskip
(i) $\|DH_p|_{E_{z_i}}\| \le C$ and

(ii) $ \|DH_p|_{E_{z_i}}\| \le c$ whenever $ z_i \in I^-_i $

\medskip \noindent
for all $p \in \Sigma_0 \setminus \mc S$. Let $i$ be fixed.
Assume $\varepsilon$ is small enough that if $\mu_e$ is an ergodic component of $\mu$
with $\mu_e\{z_i \in I^-_i\} > 1- \sqrt{\varepsilon}$, then by the Ergodic Theorem,
\[
\lambda_{z_i} \le  \sqrt{\varepsilon}  \log C + (1-\sqrt{\varepsilon}) \log c < 
\frac12 \log c < 0 \qquad \mu_e \mbox{-a.e. }
\]
{ By Proposition \ref{Prop:ConcNearToSink} and the way we constructed $\mu$, it follows that
\begin{equation}\label{Eq:Upperbndmeasure}
\mu(\{z_i\in (I_i^-)^c\})<C_\#\epsilon.
\end{equation}
}{Let $\{\mu_e\}_{e\in \mc E}$ be an ergodic decomposition of the invariant measure $\mu$,
and let $	\mc E'\subset \mc E$ consist of those $\mu_e$ satisfying
$\mu_e\{z_i \in I^-_i\} > 1- \sqrt{\varepsilon}$ {that by equation \eqref{Eq:Upperbndmeasure} is nonempty}. Then we may take $\hat \mu$ to be
$\int_{\mc E'} \mu_e$.}
\end{proof}

{Let us assume from here on that $\hat \mu$ has been normalized so that 
$\hat \mu (\Sigma_0)=1$.}
Recall that $\lambda$ 
is the positive LE of the Anosov map $A$.

\begin{corollary}\label{Cor:Exlambdalesso} There exists $\lambda_{\min}<0$ such that at  
$\hat \mu$-a.e. $p$, $\lambda$ is a LE and the other three LE are 
 $\le \lambda_{\min}$.
 \end{corollary}

\begin{proof} Let $v^u=(\partial_A^u;0,0)$ be a tangent vector at $p$ 
where $\partial^u_A\in E_A^u$ is vector along 
the unstable direction of $A$. Notice that $\forall n\in\N$, $D H^n_pv^u\in E^u_A\oplus E_{z_1z_2}$ 
and $\pi_{xy}D H^n_pv^u=A^n\partial^u_A$, so 
$\|D H^n_pv^u\|\ge \|A^n\partial^u_A\|\ge e^{\lambda n}\|v^u\|$. 
This means that there is a $DH$-invariant 1D subspace 
$E^u \subset E^u_A\oplus E_{z_1z_2}$ in which the LE
is $\ge \lambda$. This LE is in fact $= \lambda$, since by the Multiplicative Ergodic Theorem,
$\lim_{n \to \infty} \frac{1}{n} \log \sin \angle (E^u({H^n(p)}), E_{z_1z_2}) =0$ {$\hat\mu-$almost everywhere}.

A similar argument shows that $v^s := (\partial_A^s;0,0)$ grows exponentially
under $DH^{-n}$ with $\|D H^{-n}_pv^s\|\ge e^{\lambda n}\|v^s\|$. This implies that every vector 
$v \in E^s_A \oplus E_{z_1z_2}$ with a component in $E^s_A$ grows exponentially under 
$DH^{-n}$. This together with Lemma \ref{Lem:NegExp} proves 
that all three LE on $E^s_A \oplus E_{z_1z_2}$ are $\le \lambda_{\min} :=
\min\{-\lambda, \frac12 \log c\} < 0$.
\end{proof}

\subsection{$(H, \hat \mu)$ as a nonuniformly hyperbolic system}\label{Sec:UnHypSys}

The map $H:\Sigma_0 \circlearrowleft$ with the invariant measure $\hat \mu$ is, {\it a priori}, 
mildly nonuniformly hyperbolic: At $\hat \mu$-a.e. $p$, there is a splitting of its tangent space $T_p\Sigma_0 = E^u(p) {\oplus} E^s(p)$ into a 1D unstable subspace $E^u(p)$ (which {varies with}
 $p$)
and a 3D stable subspace $E^s(p) \equiv E^s_A \oplus E_{z_1,z_2}$. Restricted to $E_{z_i}$, $DH_p$
is sometimes expanding and sometimes contracting, depending on whether the $z_i$-coordinate
of $H_1(p)$ is closer to $\frac12$ or to $0$. When the $z_i$-coordinate of 
$H_1(p)$ is closer to $\frac12$, the expansion is stronger than  $e^{\lambda}$ by (A1).
A stronger expansion in $E_{z_i}$ than along $\partial^u_A$ decreases  
the angle between $E^u$ and $E_{z_1, z_2}$, and the repeated occurrence of such a scenario
can potentially cause $\angle (E^u, E_{z_1, z_2})$
to come arbitrarily close to $0$. 

We do not know that $(H, \hat \mu)$ is genuinely nonuniformly hyperbolic, but have to treat it
as such unless proven otherwise. 
A standard technique for dealing with nonuniformly hyperbolic systems is 
through the use of certain point-dependent coordinate changes called Lyapunov charts
{ (see e.g. \cite{pesin1976families,katok1980lyapunov,ledrappier1985metric,Youngergodic,katok2006invariant,blumenthal2017entropy}). We review briefly below, in nontechnical terms, what these charts 
can do for us; details are provided in Appendix A.
 
For a piecewise smooth diffeomorphism equipped with an invariant probability measure 
that is not too concentrated near the singularity set, such as $H: \Sigma_0 \circlearrowleft$,
singularity set $\mc S$, and an invariant measure $\hat \mu$ with the property in
Proposition  \ref{Prop:InvMeasBarH}, the following are known to hold:

\medskip
\begin{itemize} 
\item[(1)] Let $B(r) \subset \mathbb R^4$ be the ball of radius $r$ centered at $0$.
Then at $\hat \mu$-a.e. $p \in \Sigma_0$, there is a diffeomorphism
$$
\Phi_p : B(r(p)) \to \Sigma_0 \qquad \Phi_p(0)=p\ ,
$$
with the property that the maps
$$ 
\mc H_p = (\Phi_{H(p)})^{-1} \circ H \circ \Phi_p 
$$
that go from one chart to the next are \emph{uniformly} hyperbolic with controlled second
derivatives. In fact, $\mc H_p$ is $C^1$-near a linear map with diagonal entries 
equal to the exponentials of the Lyapunov exponents at $p$. 

\medskip
\item[(2)] In exchange for  uniform hyperbolicity, we have given up on 

(i) uniform chart sizes: $r(\cdot)$ is measurable and can be arbitrarily near zero;

(ii) uniform regularity for the chart maps $\Phi_p$. 

\medskip
\item[(3)] Chart sizes can be chosen to vary slowly along orbits, with $r(Hp)/r(p) \sim 1$. 
This ensures
the overflowing property that is crucial for establishing the existence of local stable and
unstable manifolds. Distortion estimates along unstable manifolds
are easily deduced from Lyapunov charts.

\medskip
\item[(4)]For $\hat \mu$ satisfying the condition in Proposition  \ref{Prop:InvMeasBarH}, 
it can be arranged that $\Phi_p(B(r(p)) \cap \mc S = \emptyset$ for a.e. $p$,
so that for as long as one works within the domains of charts,
one does not ``see" the singularities.
\end{itemize}}

\section{Proof of SRB property}\label{Sec:SRBProp}

Let $\nu$ be an ergodic component of the measure $\hat \mu$ defined in the last section.
We assume in particular that $\nu$ possesses all the properties of $\hat \mu$ found in Sects. \ref{Sec:LEHmu} and \ref{Sec:UnHypSys}. In this section, we will (i) show that $\nu$ is an SRB measure
for the first return map $H:\Sigma_0 \circlearrowleft$ of the flow ${\bf F}^t$;
this is carried out in Sects. \ref{Sec:ENtHnu} and \ref{Sec:SRBHnu}, and (ii) build an invariant measure
for ${\bf F}^t$ out of $\nu$; this is carried out
in Sect. \ref{Sec:SRBmeasureforflow}.

\subsection{Entropy of $(H, \nu)$}\label{Sec:ENtHnu}

Recall that $\lambda$ is the positive Lyapunov exponent of $(H, \nu)$.
Our next result concerns $h_\nu(H)$, the metric entropy of $H$ with respect to $\nu$. 

\begin{proposition}\label{Prop:EntropyEquality}
\begin{equation}\label{Eq:EntForm}
h_\nu(H)= \lambda.
\end{equation}
\end{proposition}

To prove this proposition, recall that
$H:\sigma_0\times \T^2\rightarrow \sigma_0\times \T^2$ is a skew product, which we may write as 
\[
H(x,y;z_1,z_2)=(A(x,y); T_{(x,y)}(z_1,z_2))
\]
where $\T^2_{(x,y)}$ is the vertical fiber over $(x,y)$ and
$T_{(x,y)} : \T^2_{(x,y)} \to \T^2_{A(x,y)}$ is the fiber map. Let $m_{\T^2}$ denote Lebesgue measure
on the base. Since $(\pi_{(x,y)})_*\hat\mu = m_{\T^2}$ and $(A, m_{\T^2})$ is ergodic, it follows
that $(\pi_{(x,y)})_*\nu = m_{\T^2}$. Let $\{\nu_{(x,y)}\}_{(x,y)\in\T^2}$ be a disintegration
of $\nu$ on vertical fibers, i.e.  
\[
d\nu(x,y;z_1,z_2)=d\nu_{(x,y)}(z_1,z_2)dm_{\T^2}(x,y).
\]  

\begin{lemma}\label{Lem:AtomFiber} 
$\nu_{(x,y)}$ is atomic for $m_{\T^2}$-a.e. $(x,y)$.
\end{lemma}

The idea of the proof goes back to {\cite{katok1980lyapunov}}, who proved that if all the Lyapunov exponents
of a diffeomorphism with respect to an ergodic measure are strictly negative, then the
measure is supported on a periodic orbit. A fiber version of Katok's result, meaning 
the corresponding result for skew products when all the Lyapunov exponents of 
the fiber maps are strictly negative, is proved in \cite{ruelle2001absolutely}. 
Singularities aside, our setup fits this setting, as both of the exponents of $\{T_{(x,y)}\}$ are strictly 
negative. Our proof follows that in \cite{ruelle2001absolutely} nearly verbatim. The presence of singularities
is immaterial because the proof uses Lyapunov charts, and for as
long as one works within Lyapunov charts, the singularities of $H$ are not ``visible"
by Property (ii) in Sect. \ref{Sec:UnHypSys}.

\begin{proof}[Proof of Proposition \ref{Prop:EntropyEquality}]
The assertion follows from Lemma \ref{Lem:AtomFiber} and a general result (see \cite{bogenschutz1992abramov} Corollary 2 or \cite{kifer2012ergodic}) which asserts that
the entropy of a skew-product map is equal to the sum of the entropy of the base and fiber
entropy. {(For the definition of fiber entropy, see \cite{bogenschutz1992abramov} or \cite{kifer2012ergodic}.) In our setting, denoting the fiber entropy by $h_\nu(\{T_{(x,y)}\})$, 
we have}
\[
h_\nu(H) =  h_{m_{\T^2}}(A) + h_\nu(\{T_{(x,y)}\}).
\]
Because the conditional measures on fibers are purely atomic, $h_\nu(\{T_{(x,y)}\})=0$.
\end{proof}

\subsection{Proof of SRB property for $(H,\nu)$}\label{Sec:SRBHnu}

The definition of SRB measure requires the almost-everywhere existence of unstable manifolds,
a fact guaranteed by Proposition \ref{Prop:LocManLyapCharts} for $(H,\mu)$.

One way to build SRB measures is to push forward Lebesgue measure on a curve or
disk having 
the dimension of $E^u$ and roughly aligned with $E^u$ (e.g. a piece of local unstable manifold),
and to show that for large $n$, a positive fraction of the pushed-forward measure accumulates 
on a stack of unstable manifolds of uniform length, with uniformly bounded conditional densities 
on unstable leaves. {This is an option, but one that would have to control the lengths of
the connected components of $\gamma_n$
(for which techniques are well developed in the billiards literature, see 
{\cite{chernov2006chaotic}}) and distortion along these curves} (see Sect. \ref{Sec:DistPushFor}). 
Another possibility, which we have chosen to adopt, is to appeal to a known result, namely 
converse to the entropy formula.

We recall this result, first proved for diffeomorphisms of compact manifolds (without singularities). 
Notations in the statement of Theorems \ref{Thm:EntropyFormula} and \ref{Thm:EntFormSing} and their sketches of proofs are independent
of those in the rest of this paper.

\begin{theorem}\label{Thm:EntropyFormula} [{\cite{ledrappier1984proprietes},\cite{ledrappier1985metric}}] Let $f : M \circlearrowleft$ be a $C^2$
diffeomorphism of a compact Riemannian manifold $M$, and let 
$\theta$ be an $f$-invariant Borel probability measure. Let $\lambda_i$ be the 
distinct Lyapunov exponents of $(f, \theta)$ and let $m_i$ be the multiplicity of $\lambda_i$.
In the setting above, if
\begin{equation} \label{entropyformula}
h_\theta(f) = \int \sum_{\lambda_i >0} \lambda_i m_i \ d\theta\ ,
\end{equation}
then $\theta$ is an SRB measure.
\end{theorem}

This result was first proved in {\cite{ledrappier1984proprietes}} assuming $\lambda_i \ne 0$ for all $i$;
it was extended in {\cite{ledrappier1985metric}} to allow zero Lyapunov exponents.
 
\begin{theorem} \label{Thm:EntFormSing}The setting of Theorem \ref{Thm:EntropyFormula} can be extended to the following.
Assume there is

(i) a set  $\mc S \subset M$ that is the finite union of codimension one submanifolds,

(ii)  a $C^2$-bounded map $f|_{M \setminus \mc S}: M \setminus \mc S \to M$ 
that is a diffeomorphism between 

\quad $M \setminus \mc S$ 
and its image;

(iii) an $f$-invariant Borel probability measure $\theta$ on $M \setminus \mc S$
with the property that 

\quad for some $C>0$, $\theta(\mc S_\varepsilon) < C\varepsilon$
for all small $\varepsilon>0$.

\medskip \noindent
Then the assertion in Theorem \ref{Thm:EntropyFormula} continues to be valid, that is, Eq (\ref{entropyformula})
implies $\theta$ is an SRB measure.
\end{theorem}

{ Requirement (iii) for $H:\Sigma_0\circlearrowleft$ is provided by Proposition \ref{Prop:InvMeasBarH}}.
The statement above is sufficient for our purposes, though the boundedness of 
second derivatives can be relaxed as long as $f$ is controlled
in a neighborhood of $\mc S$, and the measure can be more concentrated
near $\mc S$ than in Condition (iii) ({see e.g.  the conditions treated in
 \cite{katok2006invariant}}.

The proof of Theorem \ref{Thm:EntFormSing} is nearly identical to that of Theorem \ref{Thm:EntropyFormula}. We include { in Appendix B}
a very brief outline to show how similar the two results are and how the presence of the 
singularity set is dealt with.

\begin{corollary} $(H,\nu)$ is an SRB measure.
\end{corollary}

Ingredients of the proof include (i) the entropy formula (\ref{entropyformula}), proved
in Corollary \ref{Cor:Exlambdalesso} and Proposition \ref{Prop:EntropyEquality}; (ii) Lyapunov charts 
with the properties in { Appendix A}; and (iii) Theorem \ref{Thm:EntFormSing}, a direct application of which
 gives the desired result.


\subsection{SRB and physical measures for the flow} \label{Sec:SRBmeasureforflow}

Passing of results of this type from cross-section map to flow is standard, but we include 
it for completeness. To distinguish between objects associated with the flow 
${\bf F}^{t}$ and those associated with the return map $H$, we will write 
$W^u_{\bf F; {\rm loc}}$, resp. $W^u_{H; {\rm loc}}$, and so on.

\begin{proof}[Proof of Theorem \ref{Thm:ExistSRBforflow}] Let $\nu$ be as above, 
and let {$T_0: \Sigma_0 \to (0, \infty)$} 
 be the return time from the cross-section $\Sigma_0$ 
to itself under the flow ${\bf F}^{t}$. We let $ \nu_{\bo F}$ be the normalization of
the pushforward of $\nu$ up to return time,  i.e.,
\[
\nu_{\bo F} := \frac{1}{\int T_0 d\nu} \ \int_0^\infty  {\bf F}^t_*(  \nu|_{\{t<T_0\}}) dt
\]
where $\nu|_{\{t<T_0\}}(X)=\nu(X\cap \{t<T_0\})$ for any measurable set $X\subset \Sigma_0$.
Then $\nu_{\bo F}$ is clearly an ${\bf F}^{t}$-invariant Borel probability measure on $M$.

To prove the SRB property of $\nu_{\bo F}$, we need to show that its disintegration on
local weak unstable manifolds $W^u_{\bf F;{\rm loc}}$ of ${\bf F}^{t}$ have conditional densities.
{For $q\in M$ let $p\in \Sigma_0$ } be such that $q =  {\bf F}^{t_q}(p)$
for $0< t_q<T_0(p)$. {Then, for $\nu_{\bo F}-$a.e. $q\in M$ there  exists $r$ such that }
$$
W^u_{\bf F;{\rm loc}}(q) = \bigcup_{\{t: |t-t_q|<r\}} {\bf F}^{t}(W^u_{H; {\rm loc}}(p))
$$
 is a piece of local weak unstable manifold of the flow. From the
definition of $\nu_{\bo F}$, conditional probabilities on these objects are clearly equivalent
to $m_{W^u_{\bf F;{\rm loc}}}$, the Lebesgue measure on $W^u_{\bf F;{\rm loc}}$.
\end{proof}

We say a point $q \in M$ is {\it future-generic} with respect
to $\nu_{\bo F}$ if for all continuous observables $h: M \to \R$, 
$$
\frac{1}{T} \int_0^{T} h({\bf F}^t(q)) dt \to \int h d \nu_{\bo F} \quad \mbox{as } T \to
\infty\ .
$$
The invariant measure $\nu_{\bo F}$ is a physical measure if the set of points future-generic
with respect to $\nu_{\bo F}$ has positive Lebesgue measure ($m$)  on $M$.

\begin{proof}[Proof of Theorem \ref{Thm:ExistPhisicMeas}]
Let $W=W^u_{\bf F;{\rm loc}}(q)$, and assume $m_W$-a.e. $q' \in W$
is future-generic wrt $\nu_{\bo F}$. Defining $W^{ss}_{\bf F; {\rm loc}}(q')$ to be the
strong stable manifold at $q'$ and letting
\begin{equation} \label{Wss}
\Lambda(q):=\cup_{q' \in W} \ W^{ss}_{\bf F; {\rm loc}}(q')\ ,
\end{equation}
we have that $m(\Lambda(q))>0$ by the absolute continuity of the $W^{ss}_{\bf F; {\rm loc}}$-foliation.
Since $d({\bf F}^t(q'), {\bf F}^t(q'')) \to 0$ exponentially fast for
$q'' \in W^{ss}_{\bf F; {\rm loc}}(q')$, $q''$ is future-generic wrt $\nu_{\bo F}$ when $q'$ is.
This proves that $m$-a.e. $q'' \in \Lambda(q)$ is future-generic wrt $\nu_{\bo F}$.
 \end{proof}
 
The structures here are in fact so simple one does not need to
 invoke the absolute continuity of $W^{ss}_{\bf F; {\rm loc}}$. {We claim that for $q' \in W$
 in (\ref{Wss}), $W^{ss}_{\bf F; {\rm loc}}(q') \subset \{w=\mbox{const}\}$. 
To see this, let $p \in \Sigma_0$, and consider $W^s_{H; {\rm loc}}(p)$, the 3D local 
stable manifold for $H$ at $p$. Then $W^s_{H; {\rm loc}}(p) \subset \Sigma_0$, and 
because the return time $T_0$ is locally constant,  
\[
{\bf F}^{T_0(p)}(W^s_{H; {\rm loc}}(p)   \subset W^s_{H; {\rm loc}}(Hp) \ .
\]
Letting $T_n(p) =T_0(p)+T_0(Hp) + \cdots + \cdots T_0(H^{n-1}p)$, the
same argument gives
\[
{\bf F}^{T_n(p)}(W^s_{H; {\rm loc}}(p) )  \subset W^s_{H; {\rm loc}}(H^np)  
\]
for all $n \in \Z^+$. This implies that $W^{ss}_{\bf F; {\rm loc}}(p) \subset \{w=0\}$, and the claim follows.}

\bigskip

{
\section*{Appendix}

\noindent
{\bf A. Lyapunov charts and related results}
 
\medskip
 First some notation:
We fix a number $\lambda_0 < \min \{\lambda, -\lambda_{\min}\}$, and small numbers
$0< \delta_1, \delta_2  \ll \lambda_0$. The domains of Lyapunov charts are  subsets of 
$\R \times \R^3$, with norm $|(v,w)|':=\max\{|v|, |w|\}$ where $|\cdot|$ denotes Euclidean 
norm on $\R$ or $\R^3$, and $B(r):=\{q \in \R \times \R^3 : |q|' \le r\}$.
Norms on $\Sigma_0$ are denoted by $\|\cdot\|$ as before.
We first 
state -- without proof -- their properties, postponing explanation for some aspects to the end.

\bigskip \noindent
$\bullet$ On a set $\Gamma \subset \Sigma_0$ of full $\hat \mu$-measure are defined 

\medskip \noindent
(a) a measurable family  of linear maps $L : \Gamma \to \mc L(\R \times \R^3, T_p\Sigma_0)$
with
\[
L(p)(\R \times \{0\}) = E^u(p), \quad L(p)(\{0\} \times \R^3) = E^s(p)\ ,
\]

\smallskip \noindent
(b) a measurable function $\ell : \Gamma \to [0, \infty)$ with
\begin{equation} \label{slowvary}
e^{-\delta_1} \ell(p) \le \ell (H(p)) \le e^{\delta_1} \ell(p) \qquad \mbox{for a.e. } p\ .
\end{equation}

\medskip \noindent
$\bullet$ The Lyapunov chart at $p$ is given by 
$$
\Phi_p : B(\delta_2 \ell(p)^{-1}) \to \Sigma_0 \ , \qquad \Phi_p = L(p)|_{B(\delta_2 \ell(p)^{-1})}\ .
$$
Here we have identified neighborhoods of $0$ in $T_p\Sigma_0$ with neighborhoods of $p$ in
$\Sigma_0$ via the exponential map. Connecting maps between charts  
$$
\mc H_p : B(\delta_2 \ell(p)^{-1})) \to \R \times \R^3 \quad \mbox{are}
\quad \mc H_p = (\Phi_{H(p)})^{-1} \circ H \circ \Phi_p\ .
$$

\medskip \noindent
$\bullet$ The linear maps $L(p)$, hence $(D\Phi_p)_0$ and $(D\mc H_p)_0$, are designed to 
produce the following one-step hyperbolicity:
\begin{itemize}
\item[(i)] for $v \in \R \times \{0\}$, $|D(\mc H_p)_0(v)| \ge e^{\lambda_0}|v|$,

for $v \in \{0\} \times \R^3$, $|D(\mc H_p)_0(v)| \le e^{-\lambda_0}|v|$.
\end{itemize}

\medskip \noindent
$\bullet$ The restriction of $\Phi_p$ to $B(\delta_2 \ell(p)^{-1})$ ensures the following:
\begin{itemize}
\item[(ii)] $\Phi_p(B(\delta_2 \ell(p)^{-1})) \cap \mc S = \emptyset$;
\item[(iii)] for all $q, q' \in B(\delta_2 \ell(p)^{-1})$, 
\[
\|\Phi_p(q)-\Phi_p(q')\| \le |q-q'|' \le \ell(p) \|\Phi_p(q)-\Phi_p(q')\|\ ;
\]
\item[(iv)] Lip$(\mc H_p - (D\mc H_p)_0) \le \delta_2$, and
\item[(v)] Lip$(D\mc H_p) \le \ell(p)$.
\end{itemize}

\bigskip \noindent
{\bf Property (ii) and the slowly varying property of $\ell$}

\medskip
Property (ii) and property (\ref{slowvary}) of $\ell$ are used  to ensure the overflowing 
condition needed in the proof of local unstable manifolds for the connecting maps $\mc H_p$. 
Charts for maps with singularities were treated in {\cite{katok2006invariant}}, but since the setting of {\cite{katok2006invariant}} is more
complicated than the one here and this part of the theory is less
standard we review the main ideas on how to arrange for (\ref{slowvary}) and property (ii).

Given a measurable function $\ell_1: \Gamma \to [1,\infty)$,
we construct $\ell$ with the property in (\ref{slowvary}) by letting
\begin{equation} \label{slow2}
\ell(p) \ = \ \sup_{n \in \Z} \ e^{-\delta_1 |n|} \ell_1(H^np) 
\end{equation}
provided the right side is finite for a.e. $p$. Let us assume, without proof, that a function $\ell_0$ satisfying
\[
e^{-\delta_1} \ell_0(p) \le \ell_0 (H(p)) \le e^{\delta_1} \ell_0(p) 
\]
and all the properties above except for Item (ii) have  been constructed. Let 
\[
\ell_1(p) = \max\{\ell_0(p), d(p, \mc S)^{-1}\}\]
where $d(p, \mc S)$ is distance to the singularity set, and define $\ell(p)$ as in (\ref{slow2}).
To check the finiteness of the right side, we observe that
\[
\ell_1(H^np) \le \max\{\ell_0(H^n(p)), e^{-\delta_1 |n|} d(H^np, \mc S)^{-1}\},
\]
and the quantities $e^{-\delta_1 |n|} d(H^np, \mc S)^{-1}, n \in \Z$, are 
uniformly bounded because 
\[
\sum_{n=0}^\infty \hat \mu (\mc S_{e^{-n\delta_1/2}}) < \infty
\]
by Proposition \ref{Prop:InvMeasBarH}, so by the Borel-Cantelli Lemma, 
for $\hat \mu$-a.e. $p$, there is $n(p)$ such that 
$H^{|n|}(p) \not \in \mc S_{e^{-n\delta_1/2}}$ for all $n \ge n(p)$.

It remains to check that  $\ell$ so defined satisfies Item (ii): Since $\ell(p) \ge d(p, \mc S)^{-1}$,
it follows that for all $p' \in \Phi_p(B(\delta_2 \ell(p)^{-1}))$,
$\|p-p'\| \le \delta_2 \ell(p)^{-1} \le \delta_2 d(p, \mc S)$.

\bigskip \noindent
{\bf Local unstable manifolds and distortion} 

\medskip
The following results gleaned from Lyapunov charts are central to the definition of SRB measures. 
We state them without proof as
they are standard; see {the references above}. Below we
write the domain of charts as $B(r) = B^u(r) \times B^s(r)$ where
$B^u(r)=\{v \in \R : |v| < r\}$ and $B^s(r) = \{v \in \R^3: | v| < r\}$.

\begin{proposition}\label{Prop:LocManLyapCharts} For $\delta_2>0$ small enough, there are constants $D_\ell$
(depending on $\ell$) with respect to which the following holds at 
$\hat \mu$-a.e. $p$: 
\begin{itemize}
\item[(a)] {\rm (Existence of local unstable manifolds)} There is a $C^2$ function 
\[
\beta_p^u: B^u(\delta_2 \ell(p)^{-1}) \to B^s(\delta_2 \ell(p)^{-1})
\]
with the property that 

-- $\beta_p^u(0)=0, \ D\beta_p^u(0) = 0$ and $\|D\beta_p^u\| <\frac 1{10}$;

-- $(\mc H_{H^{-1}p})^{-1}\mbox{graph}(\beta_p^u) \subset \mbox{graph}(\beta^u_{H^{-1}p})$

-- for all $q_1, q_2 \in \mbox{graph}({\beta_p^u})$, $|(\mc H_{H^{-1}p})^{-1}q_1- (\mc H_{H^{-1}p})^{-1}q_2|'
< e^{-\lambda_0} |q_1-q_2|'$\ .

\medskip \noindent
An analogous statement holds for local stable manifolds.
\item[(b)] {\rm (Distortion estimate)} For $q \in \mbox{graph}(\beta^u_p)$, let
 $E^u(q)$ denote the tangent space
 of graph$(\beta^u_p)$ at $q$. Then for all $q_1, q_2 \in \mbox{graph}(\beta^u_p)$ and all $n \ge 1$,
\[
\left| \log \frac{D\mc H^n_{H^{-n}p}|_{E^u(q_1)}}{D \mc H^n_{H^{-n}p}|_{E^u(q_2)}} \right|
\le D_{\ell(p)} |q_1-q_2|\ .
\]
Moreover, the function
\[
\Delta_p (q) := \lim_{n \to \infty} \frac{D\mc H^n_{H^{-n}p}|_{E^u(p)}}{D \mc H^n_{H^{-n}p}|_{E^u(q)}}
\]
is Lipschitz-continuous and bounded from above and below.
\end{itemize}
\end{proposition}

We define the $\Phi_p$-image of graph$(\beta^u_p)$, denoted $W^u_{\rm loc}(p)$,
as the {\it local unstable manifold} at $p$. 
Local stable and unstable manifolds vary in size and can be arbitrarily small
in diameter, but for $p$ in {\it uniformity sets},
i.e., sets of the form $\Gamma_{\bar \ell} := \{\ell \le \bar \ell\}$ where $\bar \ell$ is a fixed number,
they contain disks of fixed radii depending on $\bar \ell$.

\bigskip \noindent
{\bf B. Outline of proof of Theorem \ref{Thm:EntFormSing}} following
\cite{ledrappier1984proprietes}

\medskip
Let  $(f, \theta)$ be as in the statement of the theorem. We construct Lyapunov charts with 
the properties in Appendix A. Note in particular Item (ii), which ensures that the images of
charts do not meet $\mc S$. This, we claim, is all that is needed to ensure that the argument
in {\cite{ledrappier1984proprietes}} will go through. We summarize very briefly this argument to give some idea of 
what it entails:

\medskip \noindent
{\bf Step 1.} One constructs a measurable partition $\eta$  with the properties that

(i) it is subordinate to unstable manifolds, i.e., for $\nu$-a.e. $p$, $\eta(p) \subset W^u(p)$ 
and 

\qquad contains a neighborhood of $p$ in $W^u(p)$, and

(ii) $\eta$ is a Markov partition, i.e., $\eta \le f^{-1}\eta$, 

\noindent
and shows that $h_\theta(f) = h_\theta(f, \eta)$ via an auxiliary finite-entropy partition {(see also \cite{ledrappierstrelcynproof})}.

The partition $\eta$ is constructed by taking a stack of local unstable disks 
$\{D^u_\alpha\}$ through points in a uniformity set $\Gamma_{\ell_0}$ (as defined at the end of
 Sect. \ref{Sec:UnHypSys}), iterating forward and taking intersections.

\medskip \noindent
{\bf Step 2.} Consider the quotient $M/\eta$. Let  $\theta_T$ be the quotient measure on this space and $\{\theta_\eta\}$
a family of conditional probabilities on elements of $\eta$. We introduce a new 
measure $\Theta$ so that $\Theta_T=\theta_T$ and for $\theta$-a.e. $p$, $\Theta_\eta$ on
$\eta(p)$ has a density equal to $\Delta_p(\cdot)$ normalized where $\Delta_p(\cdot)$ 
is as defined in Proposition \ref{Prop:LocManLyapCharts}(b). One then proves, using the equality in (\ref{entropyformula}) 
and an argument relying on the the convexity of $-\log$, that $\Theta=\theta$, so $\theta$ 
is an SRB measure.

\medskip
As can be seen from the outline above, all the structures involved in the proof originate from
within Lyapunov charts: The $D^u_\alpha$ are local unstable manifolds obtained from charts.
The putative
conditional densities $\Delta_p(\cdot)$ are defined on elements of $\eta$ (which can be 
much larger than charts),
but for $\theta$-a.e. $p$, $f^{-n}(\eta(p)) \subset W^u_{\rm loc}(q)$ for some 
$q \in \Gamma_{\ell_0}$, so again it suffices to have distortion estimates for local 
manifolds within charts as in Proposition \ref{Prop:LocManLyapCharts}(b).
In particular, once property (ii) in Sect. \ref{Sec:UnHypSys} is ensured, the singularity set does not appear in these constructions, except to render
the elements of $\eta$ more cut up, but that is immaterial.

}

\bibliographystyle{amsalpha} 
\bibliography{Untitled}

\end{document}